\documentclass[12pt]{amsart}
\usepackage{amsfonts, amssymb,amsmath, amsthm, amsxtra, latexsym, mathrsfs, amscd}
\usepackage{amsmath, amssymb, amsthm, times, float}
\usepackage{mathtools, amssymb, amsthm, times, float}
\usepackage{graphics}
\usepackage{MnSymbol}
\usepackage{enumerate}
\usepackage{enumitem}
\usepackage{tikz-cd}

\usepackage{amssymb,hyperref}
\usepackage{backref}
\usepackage[latin1]{inputenc} %

\newtheorem{theo+}{Theorem}[section]
\newtheorem{prop+}[theo+]{Proposition}
\newtheorem{coro+}[theo+]{Corollary}
\newtheorem{lemm+} [theo+]{Lemma}
\newtheorem{deep+}  [theo+]  {Deep Result}
\newtheorem{fact+}  [theo+]  {Fact}

\theoremstyle{definition}
\newtheorem{exam+}  [theo+]  {Example}
\newtheorem{rema+}  [theo+]  {Remark}
\newtheorem{defi+}  [theo+]  {Definition}
\newtheorem*{definition*}{Definition}
\newtheorem{xca+}[theo+]{Exercise}
\newtheorem{thmalph}{Theorem}

\newenvironment{proposition}{\begin{prop+}}{\end{prop+}}

\newenvironment{lemma}{\begin{lemm+}}{\end{lemm+}}

\newenvironment{remark}{\begin{rema+}}{\end{rema+}}

\numberwithin{equation}{section}

\usepackage{color}

\newcommand\beq{\begin{equation}\label}
\newcommand\eeq{\end{equation}}

\renewcommand\l{\left}

\renewcommand\a[1]{{\acute{#1}}}

\newcommand\e[1]{{\ddot{#1}}}

\def\draft{\centerline{(Draft {\the \day}/{\the\month} \the \year.)}}
\bigskip

\def\refn#1.#2{\expandafter\def\csname#1\endcsname{[#2]}}
\def\refnr#1.{\csname#1\endcsname}

\def\fe{\mathfrak e}

\def\fg{\mathfrak g}
\def\fgl{\mathfrak {gl}}
\def\fk{\mathfrak k}

\def\fl{\mathfrak l}
\def\fm{\mathfrak m}

\def\fsl{\mathfrak {sl}}
\def\ft{\mathfrak t}

\def\fu{\mathfrak u}

\def\fsp{\mathfrak{sp}}

\def\fsu{\mathfrak{su}}

\def\a{\alpha}

\def\Claminv2{|C(\Lambda)|^{-2}}

\def\varepsi{\varepsilon}

\def\de{d\varepsilon}

\def\Aa2D{A^{\a,2}(D)}
\def\bAa2D{\overline{A^{\a,2}(D)}}
\def\Ab2D{A^{\beta,2}(D)}
\def\bAb2D{\overline{A^{\beta,2}(D)}}

\def\abs#1{\vert#1\vert}

\def\Norm#1_#2{\Vert#1\Vert_{#2}}

\def\phipl12{\phi_{p_{l_1}, p_{l_2}}}
\def\phip01{\phi_{p_{0}, p_{0}}}

\def\a{\alpha}

\def\Claminv2{|C(\Lambda)|^{-2}}

\def\varepsi{\varepsilon}

\def\ad{\operatorname{ad}}
\def\Ad{\operatorname{Ad}}

\def\det{\operatorname{det}}

\def\diag{\operatorname{diag}}

\def\Ind{\operatorname{Ind}}
\def\diag{\operatorname{diag}}

\def\exp{\operatorname{exp}}

\def\sgn{\operatorname{sgn}}

\def\tr{\operatorname{tr}}

\def\Ker{\operatorname{Ker}}

\def\de{d\varepsilon}

\def\Aa2D{A^{\a,2}(D)}
\def\bAa2D{\overline{A^{\a,2}(D)}}
\def\Ab2D{A^{\beta,2}(D)}
\def\bAb2D{\overline{A^{\beta,2}(D)}}

\def\phipl12{\phi_{p_{l_1}, p_{l_2}}}
\def\phip01{\phi_{p_{0}, p_{0}}}

\def\alg/{algebra}

\def\Alg/{Algebra}

\def\alt/{alternative} 
\def\anal/{analytic}
\def\analfunc/{\anal/\ \func/}
\def\Ans/{\it Answer. \normal}
\def\ass/{associative}
\def\nass/{non-\ass/}
\def\autom/{automorphism}
\def\homom/{homomorphism}
\def\isom/{isomorphism}
\def\bdd/{bounded}
\def\Bdd/{Bounded}
\def\bddsymdom/{bounded \sym/ \dom/}
\def\Cartdom/{Cartan \dom/}
\def\bdry/{boundary}
\def\bsd/{\bdd/ \symdom/}
\def\bv/{boundary value}
\def\cf/{{\it cf}\.}
\def\Cf/{{\it Cf}\.}
\def\charr/{character}
\def\coeff/{coefficient}
\def\comm/{commutative}
\def\cpct/{compact}
\def\compl/{complex}
\def\comp/{complex}
\def\Comp/{Complex}
\def\conf/{conformal}
\def\conj/{conjugate}
\def\conn/{connect}
\def\cont/{continuous}
\def\conv/{converge} 
\def\convc/{convergence}
\def\convt/{convergent}
\def\convx/{convex}
\def\coord/{coordinate}
\def\lcoord/{local coordinate}
\def\Corr/{Corresponding}
\def\corr/{corresponding}
\def\corrd/{correspond}
\def\cov/{covariant}
\def\decomp/{decomposition}
\def\deco/{decompose}
\def\diff/{different} 
\def\Diff/{Different} 
\def\dimn/{dimension} 
\def\distr/{distribution} 
\def\div/{diverge} 
\def\dom/{domain}
\def\eg/{\hbox{\it e.g}\.}
\def\eigenf/{eigen\-\func/}
\def\eigensp/{eigen\-space}
\def\eigenv/{eigen\-value}
\def\eq/{equation}
\def\equa/{equation}
\def\de/{\diff/ial \equa/}
\def\do/{\diff/ial operator}
\def\ode/{ordinary \de/}
\def\pde/{partial \de/}
\def\pdo/{partial \diff/ial operator}
\def\psdo/{pseudo \diff/ial operator}
\def\fin/{finite}
\def\Ex/{\it Example.\ \normal}
\def\Exnr#1/{\it Example #1.\ \normal}
\def\foll/{follow}
\def\follg/{following}
\def\Follg/{Following}
\def\func/{function}
\def\Func/{Function}
\def\Fonc/{Fonc\-tion}
\def\fonc/{fonc\-tion}
\def\Funk/{Funk\-tion}
\def\funk/{Funk\-tion}
\def\gen/{general}
\def\har/{harmonic}
\def\Hint/{\it Hint. \normal}
\def\hist/{historic}
\def\histcl/{historical}
\def\hol/{holo\-morphic}
\def\homog/{ho\-mo\-ge\-ne\-ous}
\def\hyp/{hyper\-bolic}
\def\hyperg/{hyper\-geometric}
\def\ie/{\hbox{\it i.e.}}
\def\iff/{if and only if}
\def\ineq/{inequality}
\def\infra/{{\it inf\-ra}}
\def\ultra/{{\it ult\-ra}}
\def\Inpart/{In particular}
\def\inpart/{in particular}
\def\instof/{instead of}
\def\interps/{interpolation space}
\def\interp/{interpolation}
\def\Interp/{Interpolation}
\def\interpr/{Interpretation}
\def\Intr/{Introduction}
\def\intv/{interval}
\def\inv/{invariant}
\def\invc/{invariance}
\def\Iowords/{In other words}
\def\iowords/{in other words}
\def\ipr/{inner product}
\def\irred/{irreducible}
\def\lb/{line bundle}
\def\lin/{linear}
\def\lhs/{left hand side}
\def\rhs/{right hand side}
\def\loc/{local}
\def\math/{mathematic}
\def\mathcn/{\math/ian}
\def\manif/{manifold}
\def\meas/{measure}
\def\measl/{measurable}
\def\mero/{mero\-morphic}
\def\mon/{monomial}
\def\monog/{monogenic}
\def\mult/{multiple}
\def\multy/{multiply}
\def\multn/{multiplication}
\def\nas/{necessary and sufficient}
\def\nbd/{neighborhood}
\def\neg/{negative}
\def\nondeg/{nondegenerate}
\def\Oohand/{On the other hand}
\def\oohand/{on the other hand}
\def\Oonhand/{On the one hand}
\def\oonhand/{on the one hand}
\def\oper/{operator}
\def\orth/{ortho\-gonal}
\def\orthon/{ortho\-normal}
\def\otoh/{on the other hand}
\def\quat/{quaternion}
\def\pp/{\hbox{a. e.}}
\def\psh/{plurisubharmonic}
\def\pol/{polynomial}
\def\pot/{potential}
\def\pos/{positive}
\def\princ/{principle}
\def\prob/{probability}
\def\proj/{projective}
\def\projn/{projection}
\def\Proof/{\it Proof:\normal}
\def\Rem/{\it Remark\normal}
\def\Remnr#1/{\it Remark\ \normal #1. }
\def\rep/{representation}
\def\reps/{representations}
\def\meta/{metaplectic representation}
\def\repr/{reproducing}
\def\reprker/{reproducing kernel}
\def\resp/{respective} 
\def\resply/{respectively}
\def\restr/{restriction}
\def\sa/{self-adjoint}
\def\st/{such that}

\def\sol/{solution}
\def\ru/{space}
\def\sph/{spherical}
\def\ssp/{sub\ru/}
\def\sym/{symmetric}
\def\Sym/{Symmetric}
\def\symb/{symbol}
\def\symbc/{symbolic}
\def\symdom/{\sym/ domain}
\def\symp/{symplectic}
\def\Theor#1/{\fet Theorem #1.\ \normal}
\def\Lem#1/{\fet Lemma #1.\ \normal}
\def\Lemma/{\fet Lemma.\ \normal}
\def\topl/{topology}
\def\topll/{topological}
\def\transf/{transform}
\def\transl/{translation}
\def\transfn/{transformation}
\def\transv/{transvectant}
\def\trig/{trigonometric}
\def\tril/{trilinear}
\def\trilf/{trilinear form}
\def\uhp/{upper halfplane}
\def\uhs/{upper halfspace}
\def\vb/{vector bundle}
\def\vf/{vector field}
\def\vsp/{vector space}
\def\wrt/{with respect to}
\def\Wlog/{Without loss of generality}
\def\a{\alpha}

\def\e{\varepsilon}

\def\Ab/{Abel}
\def\Ban/{Banach}
\def\Bansp/{\Ban/ space}
\def\Belt/{Bel\-tra\-mi}
\def\Berg/{Berg\-man}
\def\Bern/{Ber\-nou\-lli}
\def\Berz/{Berezin}
\def\Bess/{Bessel}
\def\Cart/{Car\-tan}
\def\Cay/{Cay\-ley}
\def\CG/{Clebsch-Gordan}
\def\Cl/{Clifford}
\def\CR/{Cauchy-Rie\-mann}
\def\Dir/{Dirichlet}
\def\Eucl/{Euclide}
\def\Eucln/{Euclidean}
\def\F/{Fourier}
\def\Hank/{Hankel}
\def\Hankf/{\Hank/ form}
\def\Herm/{Hermite}
\def\Hilb/{Hilbert}
\def\Hilbs/{Hilbert space}
\def\Hilbsp/{Hilbert space}
\def\HS/{Hilbert-Schmidt}
\def\Lag/{La\-grange}
\def\Lap/{La\-place}
\def\LapBelt/{\Lap/-\Belt/}
\def\Leb/{Lebesgue}
\def\Marc/{Mar\-cin\-kie\-wicz}
\def\Moeb/{Moebius}
\def\Moebt/{Moebius transformation}
\def\Moebtransfn/{Moebius transformation}
\def\Pla/{Plan\-che\-rel}
\def\Poin/{Poin\-car\'e}
\def\Riem/{Rie\-mann}
\def\Riemn/{\Riem/ian}
\def\psRiemn/{pseudo-\Riem/ian}
\def\Riems/{Rie\-mann surface}
\def\Schroe/{Schr\"odinger}
\def\Weier/{Weier\-strass}

\def\im{\operatorname{Im}}
\def\anal/{analytic}
\def\bsd/{bounded symmetric domain  }
\def\bdd/{bounded}
\def\calc/{calculation}\def\conj{conjugate}
\def\calci/{calculating}\def\eg{e.g.}
\def\conj/{conjugate}
\def\deco/{decomposition}
\def\eg/{e.g.}
\def\fct/{function}
\def\gp/{group}
\def\hw/{highest weight}
\def\hwv/{highest weight vector}
\def\hwvs/{highest weight vectors}
\def\lw/{lowest weight}
\def\lwv/{lowest weight vector}
\def\lwvs/{lowest weight vectors}
\def\hds/{holomorphic discrete series}
\def\iff/{if and only if}
\def\inv/{invariant}
\def\irrde/{irreducible decomposition}
\def\meas/{measure}
\def\transf/{transform}
\def\rep/{representation}
\def\resp/{respectively}
\def\inters/{intertwines}
\def\interg/{intertwining}
\def\meta/{metaplectic representation}
\def\qu/{quaternion}
\def\rep/{representation}
\def\symdom/{ symmetric domain}
\def\st/{such that}
\def\shd/{subhead}
\def\transf/{transform}
\def\wrt/{with respect to}

\def\Norm#1#2#3{\Vert#1\Vert^{#3}_{{#2}}}

\def\tr{\operatorname{tr}}

\newmuskip\pFqmuskip

\newcommand*\pFq[6][8]{%
	\begingroup 
	\pFqmuskip=#1mu\relax
	\mathchardef\normalcomma=\mathcode`,
	\mathcode`\,=\string"8000
	\begingroup\lccode`\~=`\,
	\lowercase{\endgroup\let~}\pFqcomma
	{}_{#2}F_{#3}{\left(\genfrac..{0pt}{}{#4}{#5};#6\right)}%
	\endgroup
}
\newcommand{\pFqcomma}{{\normalcomma}\mskip\pFqmuskip}

\DeclareMathOperator{\SL}{SL}
\let\sl\relax
\DeclareMathOperator{\sl}{\mathfrak{sl}}
\DeclareMathOperator{\upO}{O}

\DeclareMathOperator{\SO}{SO}
\DeclareMathOperator{\so}{\mathfrak{so}}
\DeclareMathOperator{\Sp}{Sp}

\let\sp\relax
\DeclareMathOperator{\sp}{\mathfrak{sp}}
\DeclareMathOperator{\SU}{SU}
\DeclareMathOperator{\su}{\mathfrak{su}}

\DeclareMathOperator{\spin}{\mathfrak{spin}}
\DeclareMathOperator{\upU}{U}

\newcommand{\fraka}{\mathfrak{a}}

\newcommand{\frake}{\mathfrak{e}}

\newcommand{\frakg}{\mathfrak{g}}

\newcommand{\frakk}{\mathfrak{k}}
\newcommand{\frakl}{\mathfrak{l}}
\newcommand{\frakm}{\mathfrak{m}}
\newcommand{\frakn}{\mathfrak{n}}
\newcommand{\frakp}{\mathfrak{p}}

\newcommand{\fraks}{\mathfrak{s}}
\newcommand{\frakt}{\mathfrak{t}}
\newcommand{\fraku}{\mathfrak{u}}
\newcommand{\frakz}{\mathfrak{z}}

\newcommand{\CC}{\mathbb{C}}

\newcommand{\NN}{\mathbb{N}}

\newcommand{\RR}{\mathbb{R}}
\newcommand{\ZZ}{\mathbb{Z}}

\newcommand{\calD}{\mathcal{D}}

\newcommand{\calF}{\mathcal{F}}
\newcommand{\calH}{\mathcal{H}}

\newcommand{\calO}{\mathcal{O}}
\newcommand{\calP}{\mathcal{P}}

\newcommand{\calS}{\mathcal{S}}

\newcommand{\calU}{\mathcal{U}}
\newcommand{\calV}{\mathcal{V}}

\DeclareMathOperator{\Hom}{Hom}

\DeclareMathOperator{\const}{const}

\DeclareMathOperator{\met}{met}

\renewcommand\Re{\operatorname{Re}}

\renewcommand{\min}{{\textup{min}}}
\DeclareMathOperator{\otimeshat}{\hat{\otimes}}
\DeclareMathOperator{\GKDIM}{GKDIM}
\newcommand{\ntm}{\textup{ntm}}

\begin{document}
	
\title[Heisenberg parabolically induced representations of Hermitian Lie groups, Part II]{Heisenberg parabolically induced representations of Hermitian Lie groups, Part II: Next-to-minimal representations and branching rules}

\begin{abstract}
Every simple Hermitian Lie group has a unique family of spherical representations induced from a maximal parabolic subgroup whose unipotent radical is a Heisenberg group. For most Hermitian groups, this family contains a complementary series, and at its endpoint sits a proper unitarizable subrepresentation. We show that this subrepresentation is next-to-minimal in the sense that its associated variety is a next-to-minimal nilpotent coadjoint orbit. Moreover, for the Hermitian groups $\SO_0(2,n)$ and $E_{6(-14)}$ we study some branching problems of these next-to-minimal representations.
\end{abstract}



\author{Jan Frahm}
\address{Department of Mathematics, Aarhus University, Ny Munkegade 118, 8000 Aarhus C, Denmark}
\email{frahm@math.au.dk}

\author{Clemens Weiske}
\address{Mathematical Sciences, Chalmers University of Technology and Mathematical Sciences, G\"oteborg University, SE-412 96 G\"oteborg, Sweden}
\email{weiske@chalmers.se}

\author{Genkai Zhang}
\address{Mathematical Sciences, Chalmers University of Technology and Mathematical Sciences, G\"oteborg University, SE-412 96 G\"oteborg, Sweden}
\email{genkai@chalmers.se}

\thanks{The first named author was supported by a research grant from
  the Villum Foundation (Grant No. 00025373), the second named author
  was supported by a research grant from the Knut and Alice Wallenberg
  foundation (KAW 2020.0275) and the third named author was supported
  partially by the Swedish Research Council (VR, Grants 2018-03402,  2022-02861).}

\maketitle

\section*{Introduction}

\emph{Minimal representations} of simple Lie groups are well studied and have several equivalent descriptions. The most natural one is by their relation to minimal nilpotent coadjoint orbits via the orbit philosophy. They are often unique and show up naturally in relation to the theta correspondence, unipotent representations, and the quantization of nilpotent coadjoint orbits. Moreover, they occur as the local archimedean components of certain automorphic representations of reductive groups over global fields. Minimal representations of real groups have been studied extensively from various different perspectives such as classical harmonic
analysis, partial differential equations, complex analysis or conformal geometry (see
e.g. \cite{BSKO12,FL12,HKMO12,HSS12,KM11,KO03a,KO03c}). From the
representation theoretic point of view, one particularly important
question in this context is how minimal representations decompose when
restricted to certain subgroups. If the subgroup arises from a dual pair,
this question falls into the framework of the celebrated theta
correspondence (see e.g. \cite{How89,HPS96,KO03b,Li99}). On the other
hand, the restriction to symmetric subgroups has also turned out to
reveal interesting new results (see e.g. \cite{Kob22,MO15}).

Much less studied are \emph{next-to-minimal representations};
these ought to correspond to nilpotent coadjoint orbits whose closure is the union of the orbit itself and the trivial and minimal orbits, see the precise
definition below. For some groups it has been shown that certain next-to-minimal representations also occur as local archimedean components of global automorphic representations (see \cite{GMV15}), and their automorphic realizations seem to be of growing interest, in particular for exceptional groups (see \cite{BP17,GGKPS21,GGKPS22,Pol22}). It is therefore desirable to gain a better understanding of next-to-minimal representations, both globally and locally. The purpose of this paper is to study some branching laws for next-to-minimal representations of Hermitian Lie groups.

While minimal representations of Hermitian Lie groups turn out to be
unitary highest or lowest weight representations, there do exist
next-to-minimal representations which are neither highest nor lowest
weight representations. This makes them more difficult to construct
and understand. In our previous work~\cite{FWZ22}, where we studied
Heisenberg parabolically induced representations of Hermitian groups, we exhibited
some interesting unitary representations showing up at the end of the
complementary series.
We proved that the representations can be realized on Hilbert spaces of distributions on the Heisenberg group whose Weyl transforms have rank one as operators on Fock spaces. Thus they have similar properties as the last point in the Wallach set for scalar unitary highest weight representations described using the Euclidean Fourier transform \cite{FK94,How82,H08}.

In this paper, we show that they are in fact next-to-minimal representations. For the Hermitian groups $\SO_0(2,n)$ and $E_{6(-14)}$ we further study corresponding branching problems when restricting these next-to-minimal representations to certain symmetric subgroups.

Let us describe our results in more detail.

\subsection*{Next-to-minimal representations}

Let $\frakg$ be a real form of a complex simple Lie algebra $\frakg^\CC$ and $(\frakg^\CC)^*$ the dual space of $\frakg^\CC$. The nilpotent cone in $(\frakg^\CC)^*$ decomposes into finitely many nilpotent coadjoint orbits, and among them there is precisely one of minimal dimension, $\calO_\min$. Note that the closure of $\calO_\min$ equals $\calO_\min\cup\{0\}$. The following definition is in the spirit of \cite{GGKPS22,GMV15}:

\begin{definition*}
A nilpotent coadjoint orbit $\calO\subseteq(\frakg^\CC)^*$ is called \emph{next-to-minimal} if its closure is equal to $\calO\cup\calO_\min\cup\{0\}$. An irreducible unitary representation $\pi$ of a Lie group $G$ with Lie algebra $\frakg$ is called \emph{next-to-minimal} if its associated variety in $(\frakg^\CC)^*$ is the closure of a next-to-minimal nilpotent coadjoint orbit.
\end{definition*}

Now let $G$ be a simple Hermitian Lie group with $\frakg\not\simeq\sp(n,\RR)$, i.e. $\frakg$ is one of the following Lie algebras:
$$ \su(p,q), \quad \so(2,n), \quad \so^*(2n), \quad \frake_{6(-14)}, \quad \frake_{7(-25)}. $$
Up to conjugation, $G$ has a unique maximal parabolic subgroup $P=MAN$ whose unipotent radical $N$ is a Heisenberg group. We consider the degenerate principal series representations
$$ \pi_\nu = \Ind_P^G(1\otimes e^\nu\otimes 1) \qquad (\nu\in(\fraka^\CC)^*), $$
where $\fraka$ denotes the Lie algebra of $A$. Here, $\pi_\nu$ is
normalized such that $\pi_\nu$ contains the trivial representation as
a quotient for $\nu=\rho$ and as a subrepresentation for $\nu=-\rho$, $\rho$ being the half sum of positive roots with respect to $\fraka$, and $\pi_\nu$ is unitary for $\nu\in i\fraka^*$. Excluding the case $\frakg\simeq\su(p,q)$ with $p-q$ odd, there exists by \cite[Theorem 4.1]{FWZ22} an interval $(-\nu_0,\nu_0)\subseteq\fraka^*$ such that $\pi_\nu$, $\nu\in\fraka^*$, is irreducible and unitarizable if and only if $\nu\in(-\nu_0,\nu_0)$. Let $A_\nu:\pi_\nu\to\pi_{-\nu}$ denote the Knapp--Stein standard intertwining operators.

\begin{thmalph}[see Section~\ref{sec:AssociatedVarieties}]
	$\pi_\ntm=\Ker A_{-\nu_0}=\im A_{\nu_0}\subseteq\pi_{-\nu_0}$ is a proper irreducible and unitarizable subrepresentation which is spherical and next-to-minimal. Its $K$-type decomposition is given in Theorem~\ref{thm:CS_submodule}.
\end{thmalph}

For $G=\upO(p,q)$ with $\min(p,q)\geq4$, a next-to-minimal
representation was constructed in \cite{ZH97}, but their construction
does not extend to the case $\min(p,q)=2$. The reason is that
$\upO(p,q)$ has two different next-to-minimal nilpotent coadjoint
orbits. For $\min(p,q)=2$ the next-to-minimal orbit considered in \cite{ZH97} does not have real points, so there cannot exist a unitary representation whose associated variety equals this orbit. Our next-to-minimal representation has the other next-to-minimal orbit as associated variety.

We further remark that whenever the rank of $G$ is at least $3$, the
analytic continuation of the scalar type holomorphic discrete series
of $G$ contains a next-to-minimal representation, namely the one
corresponding to the next-to-minimal discrete point in the Wallach
set. Our representation is different from this one as it is neither
highest nor lowest weight module.

\subsection*{Branching $G\searrow\SL(2,\RR)\times M$}

If $P=MAN$ is a Langlands decomposition of the Heisenberg parabolic subgroup $P$ of $G$, then the centralizer of $M$ in $G$ is a subgroup locally isomorphic to $\SL(2,\RR)$. In fact, the two subgroups $M$ and $\SL(2,\RR)$ form a dual pair inside $G$. We study the restriction of $\pi_\ntm$ to $\SL(2,\RR)\times M$ for the two cases $G=\SO_0(2,n)$ and $G=E_{6(-14)}$.

Let us first consider $G=\SO_0(2,n)$, then
$M=\SL(2,\RR)\times\SO(n-2)$. In fact, the subgroup $\SL(2,\RR)$ of
$M$ is conjugate to the centralizer of $M$, and we simply write
$\SL(2,\RR)\times M$ as $\SL(2,\RR)\times\SL(2,\RR)\times\SO(n-2)$
without distinguishing between the two copies of $\SL(2,\RR)$. To
state the decomposition of $\pi_\ntm$ let $\eta_k^{\SO(n-2)}$ denote
the irreducible representation of $\SO(n-2)$ on the space
$\calH^k(\RR^{n-2})$ of harmonic homogeneous polynomials on
$\RR^{n-2}$ of degree $k$. Moreover, for $\mu\in i\RR$ and $\varepsilon\in\ZZ/2\ZZ$ let $\tau_{\mu,\varepsilon}^{\SL(2,\RR)}$ be the unitary principal series of $\SL(2,\RR)$, spherical for $\varepsilon=0$ and non-spherical for $\varepsilon=1$, and for $\ell\in\ZZ$, $|\ell|\geq2$ let $\tau_\ell^{\SL(2,\RR)}$ be the discrete series of $\SL(2,\RR)$ of parameter $\ell$.

\begin{thmalph}[see Theorem~\ref{thm:RestrictionOfSO(2,n)toMxSL2}]\label{thm:IntroThmC}
	The restriction of the next-to-minimal representation $\pi_\ntm$ of $G=\SO_0(2,n)$ to $\SL(2,\RR)\times\SL(2,\RR)\times\SO(n-2)$ is given by
	\begin{multline*}
		\pi_\ntm|_{\SL(2,\RR)\times\SL(2,\RR)\times\SO(n-2)} \simeq \bigoplus_{k=0}^\infty\Bigg(\int_{i\RR_{\geq0}}\tau_{\mu,\operatorname{par}(\mu)}^{\SL(2,\RR)}\boxtimes\tau_{\mu,\operatorname{par}(\mu)}^{\SL(2,\RR)}\,d\mu\\
		\oplus\bigoplus_{\substack{2\leq|\ell|\leq k\\\ell\equiv k\mod2}}\tau_\ell^{\SL(2,\RR)}\boxtimes\tau_\ell^{\SL(2,\RR)}\Bigg)\boxtimes\eta_k^{\SO(n-2)}.
	\end{multline*}
\end{thmalph}

Now let $G=E_{6(-14)}$, then $M=\SU(5,1)$. Let $P_M=M_MA_MN_M$ be a minimal parabolic subgroup of $M$. Then $M_M$ is a double cover of $\upU(4)$ whose irreducible representations are parameterized by tuples $(\nu_1,\nu_2,\nu_3,\nu_4)$ with $\nu_j\in\frac{1}{2}\ZZ$ and $\nu_i-\nu_j\in\NN$ for all $1\leq i<j\leq4$. Moreover, the irreducible unitary characters of $A_M$ are parameterized by $\mu\in i\RR$. We denote the corresponding parabolically induced representation of $\SU(5,1)$ by $\tau_{(\nu_1,\nu_2,\nu_3,\nu_4),\mu}^{\SU(5,1)}$. Keeping the notation for representations of $\SL(2,\RR)$ from the previous discussion, we can identify part of the spectrum for the decomposition of the restriction of $\pi_\ntm$ to $\SL(2,\RR)\times\SU(5,1)$:

\begin{thmalph}[see Theorem~\ref{thm:E6_MxSL(2)}]\label{thm:IntroThmD}
	The restriction of the next-to-minimal representation $\pi_\ntm$ of $E_{6(-14)}$ to $\SL(2, \mathbb{R})\times\SU(5,1)$ contains for every $k>0$ and every $0<m\leq k$ a representation of the form
	$$
	\int_{i\RR_{\geq0}}\tau_{-l}^{\SL(2,\RR)}\boxtimes\tau_{(\frac{k}{2},\frac{k}{2}-m,-\frac{k}{2},-\frac{k}{2}),\mu}\,d\mu \oplus \int_{i\RR_{\geq0}}\tau_{l}^{\SL(2,\RR)}\boxtimes\tau_{(\frac{k}{2},\frac{k}{2},m-\frac{k}{2},-\frac{k}{2}),\mu}\,d\mu
	$$
	for some $l\geq1$ as a direct summand.
\end{thmalph}

The proofs of both Theorem~\ref{thm:IntroThmC} and \ref{thm:IntroThmD}
make use of the realization of $\pi_\ntm$ in terms of the Heisenberg
group Fourier transform.
Indeed, since the restriction problem is mostly of analytic nature, the Fourier transform is an effective tool. It relates the decomposition of $\pi_\ntm|_M$ to the decomposition of the tensor product of some lowest/highest
weight representations of $M$ that occur in the restriction to $M$ of
a metaplectic representation. In both cases this tensor
product can be decomposed completely. Finally, the contribution of the $\SL(2,\RR)$-factor is obtained by conjugating $\SL(2,\RR)$ inside $G$ to a subgroup of $M$ and comparing with the decomposition when restricting $\pi_\ntm$ to $M$. For $G=\SO_0(2,n)$ we can identify the action of the $\SL(2,\RR)$-factor explicitly while for $G=E_{6(-14)}$ we are only able to describe its restriction to a parabolic subgroup of $\SL(2,\RR)$.\\

\subsection*{Structure of this article} We first recall the main results of our previous work~\cite{FWZ22} in Section~\ref{sec:Preliminiaries}. In Section~\ref{sec:AssociatedVarieties} we compute the Gelfand--Kirillov dimension of $\pi_\ntm$ using its $K$-type decomposition and use this information to show that the associated variety of $\pi_\ntm$ is indeed a next-to-minimal nilpotent coadjoint orbit in $(\frakg^\CC)^*$. Finally, Section~\ref{sec:BranchingGtoMxSL(2,R)} is concerned with the restriction of the next-to-minimal representation $\pi_\ntm$ of $G$ to the subgroup $\SL(2,\RR)\times M$. While Section~\ref{sec:FourierTransformForBranching} contains some general observation on how to use the Heisenberg Fourier transform for this problem, we treat the case $G=\SO_0(2,n)$ in Section~\ref{sec:BranchingSO(2,n)toMxSL(2,R)}, and in Section~\ref{sec:BranchingE6toMxSL(2,R)} we discuss the case $G=E_{6(-14)}$.

\subsection*{Notation}
For a unitarizable Casselman--Wallach representation $\pi$ we abuse notation and denote its unitary closure also by $\pi$, while usually suppressing the notation of the underlying Fr\'{e}chet resp. Hilbert spaces.
In this sense, if not stated otherwise, all direct sums and tensor products of representations are to be considered in the category of Hilbert spaces.

\section{Preliminaries}\label{sec:Preliminiaries}

We recall the results about Heisenberg parabolically induced representations of Hermitian Lie groups from \cite{Zha22} as well as their realization using the Heisenberg group Fourier transform obtained in \cite{FWZ22}.

\subsection{Heisenberg parabolically induced representations}

Let $G$ be a simple Hermitian Lie group with Lie algebra $\frakg$ not
isomorphic to $\sl(2,\RR)$. Then $G$ has a unique conjugacy class of
parabolic subgroups whose unipotent radical is a Heisenberg group. Let
$P=MAN$ be the Langlands decomposition
of one of them and write $\frakm$, $\fraka$, $\frakn$ for the Lie algebras of $M$, $A$, $N$.

The one-dimensional subalgebra $\fraka$ is spanned by an element $H\in\fraka$ such that $$\frakg=\bar{\frakn}_2\oplus \bar{\frakn}_1\oplus(\frakm\oplus \fraka)\oplus \frakn_1\oplus \frakn_2$$ is a decomposition into eigenspaces of $\ad(H)$ with eigenvalues $-2,-1,0,1,2$ and $\frakn=\frakn_1\oplus \frakn_2$. We further write $\bar{\frakn}=\bar{\frakn}_1\oplus\bar{\frakn}_2$ for the Lie algebra of the opposite unipotent radical $\bar{N}$. 
Note that $\frakn_2$ and $\bar{\frakn}_2$ equal the center of $\frakn$ and $\bar{\frakn}$ respectively and are therefore one-dimensional and we choose $E\in \frakn_2$ and $F\in \bar{\frakn}_2$ such that $[E,F]=H$. Hence $\{E,H,F\}$ form an $\sl(2)$-triple.
We identify $(\fraka^\CC)^*$ with $\CC$ by $\nu\mapsto\nu(H)$, then the half sum of positive roots is given by $\rho=d_1+1$, where $d_1=\frac{1}{2}\dim_\RR\frakn_1\in\NN$ (see Table~\ref{tab:2}). Consider the degenerate principal series representations (smooth normalized parabolic induction)
$$ \pi_\nu = \Ind_P^G(1\otimes e^\nu\otimes1) \qquad\qquad (\nu\in(\fraka^\CC)^*\simeq\CC),$$
on the space $$I(\nu)=\{f\in C^\infty(G); \; f(gman)=a^{-\nu-\rho}f(g)\, \forall man \in MAN\}.$$

To describe the $K$-type structure of $\pi_\nu$, let $\theta$ be a Cartan
involution on $G$ which leaves $MA$ invariant and maps $E$ to $-F$ and denote by
$K\subseteq G$ the corresponding maximal compact subgroup. Its Lie
algebra $\frakk$ decomposes into
$\frakk=\frakz(\frakk)
\oplus\frakk'$ with $\frakk'=[\frakk,\frakk]$ the semisimple part and $\frakz(\frakk)$ the
center.

The Lie algebra $\frakk'$ has a Hermitian symmetric pair $(\frakk', \fl')$ of rank $2$. Let $2\alpha_1>2\alpha_2$ be the Harish-Chandra strongly orthogonal roots and note that the restricted root system for the pair $(\frakk',\frakl')$ is of type $C_2$ or $BC_2$:
$$ \{\pm2\alpha_1,\pm2\alpha_2\}\cup\{\pm\alpha_1\pm\alpha_2\}\quad\Big[\cup\{\pm\alpha_1,\pm\alpha_2\}\Big]. $$
The multiplicity of $\pm2\alpha_1$ and $\pm2\alpha_2$ is $1$, and we write $a_1$ resp. $2b_1$ for the root multiplicities of $\pm\alpha_1\pm\alpha_2$ resp. $\pm\alpha_1$ and $\pm\alpha_2$. The values of $a_1$ and $b_1$ for the different Hermitian groups are listed in Table~\ref{tab:2}.

We further let $\alpha_0$ be a linear functional on the center $\frakz(\frakk^\CC)$ normalized such that it has value $1$ on the central element with eigenvalues $\pm1$ on $\frakp^\CC$. Using the notation $W_{\mu_1,\mu_2,\ell}=W_{\mu_1\alpha_1+\mu_2\alpha_2+\ell\alpha_0}$, the $K$-type decomposition of $\pi_\nu$ can be written as follows:
	\begin{align*}
		& \pi_\nu|_K\simeq
			\bigoplus_{\substack{\mu_1\geq\mu_2\geq|\ell|\\\mu_1\equiv\mu_2\equiv\ell\mod2}}W_{\mu_1,\mu_2,\ell} &&\text{if $\frakg \not\simeq\sp(n,\RR),\su(p,q)$,}\\
		& \pi_\nu|_K\simeq
		\bigoplus_{\substack{\mu_1, \mu_2\geq \abs{\ell},\\ \mu_1\equiv \mu_2 \equiv \ell \mod 2}} W_{\mu_1,\mu_2,\ell} &&\text{if $\frakg \simeq \su(p,q)$.}
	\end{align*}

\begin{table}[h!]
	\begin{center}
		\begin{tabular}{|l|l|l|l|l|}
			\hline
			$\frakg$ & $\frakk$ & $\frakm$ & $(a_1,b_1)$ & $d_1$\\
			\hline
			$\su(p,q)$ & $\fraks(\fraku(p)+\fraku(q))$ & $\fraku(p-1,q-1)$ & $(0, p-2), (0, q-2)$ & $p+q-2$\\ 
			\hline
			$\so^\ast(2n)$ & $ \fraku(n)$ & $\so^\ast(2n-4)+\su(2)$ & $(2, n-4)$ & $2n-4$\\
			\hline
			$\sp(n, \mathbb R)$ & $\fraku(n)$ & $\sp(n-1,\mathbb R)$ & $ (0, n-2)$ & $n-1$\\
			\hline
			$\so(n, 2)$ & $\so(n)+\so(2)$ & $\sl(2,\RR)+\so(n-2)$ & $(n-4, 0)$ & $n-2$\\ 
			\hline  
			$\frake_{6(-14)}$ & $\spin(10)+\so(2)$ & $\su(5,1)$ & $(4, 2)$ & $10$\\ 
			\hline
			$\frake_{7(-25)}$ & $\frake_6+\so(2)$ & $\so(10,2)$ & $(6, 4)$ & $16$\\ 
			\hline  
		\end{tabular}
		\vskip0.20cm
		\caption
		{Hermitian Lie algebras $\frakg$ with subalgebras $\frakk$ and $\frakm$ and structure constants $a_1$, $b_1$ and $d_1$}
		\label{tab:2}
	\end{center}
\end{table}

\subsection{Intertwining operators}

Let
$$ w_0=\exp\left(\frac{\pi}{2}(E-F)\right)\in K, $$
then $\Ad(w_0)P=\bar{P}=MA\bar{N}$ is the parabolic subgroup opposite to $P$. Let $A_\nu: I(\nu) \to I(-\nu)$ be the standard intertwining operator, which is for $\Re(\nu)>\rho$, $f\in I(\nu)$ and $g\in G$ given by the convergent integral
$$ A_\nu f(g)=\int_{\bar{N}}f(gw_0\bar{n})\, d\bar{n} $$
and extended meromorphically to all $\nu\in\CC$. For $\nu\in\RR$, the corresponding Hermitian form on $I(\nu)$ given by
\begin{equation}\label{eq:DefInvForm}
	\langle f_1,f_2\rangle_\nu = \langle A_\nu f_1,f_2\rangle_{L^2(\bar{N})} = \int_{\bar{N}} A_\nu f_1(\bar{n})\overline{f_2(\bar{n})}\,d\bar{n}
\end{equation}
is $G$-invariant. In \cite{FWZ22} we studied unitarizability and the composition series of the representations $\pi_\nu$ using the operators $A_\nu$. We recall the main results.

\begin{theo+}[\cite{FWZ22} Theorem~4.1 and Theorem~4.3]\label{thm:CS_submodule}
	Assume $\frakg\not\simeq\sp(n,\RR)$ and let $$\nu_0=\begin{cases}
		1 & \text{if $\frakg \simeq \su(p,q)$ and $p-q$ is even,}\\
		a_1+1 & \text{if $\frakg \not\simeq \su(p,q)$.}
	\end{cases}$$
\begin{enumerate}[label=(\roman*)]
	\item For $\nu\in\RR$, the representation $\pi_\nu$ belongs to the complementary series, i.e. it is irreducible and unitarizable, if and only if $\abs{\nu}<\nu_0$.
	\item At the endpoint $-\nu_0$ of the complementary series there is a proper unitarizable subrepresentation
	$\pi_\ntm:=\Ker A_{-\nu_0}=\im A_{\nu_0}$ with  $K$-type decomposition
	\begin{align*}
		& \pi_\ntm|_K\simeq
			\bigoplus_{\substack{\mu \geq|\ell|\\\mu\equiv\ell\mod2}}W_{\mu,\mu,\ell} &&\text{if $\frakg \not\simeq \su(p,q)$,}\\
		& \pi_\ntm|_K\simeq
		\bigoplus_{\substack{\mu_1, \mu_2\geq \abs{\ell},\\ \mu_1-\mu_2=q-p, \\ \mu_1\equiv \mu_2 \equiv \ell \mod 2}} W_{\mu_1,\mu_2,\ell} &&\text{if $\frakg \simeq \su(p,q)$, $p-q$ even.}
	\end{align*}
\end{enumerate}
\end{theo+}

\subsection{The Heisenberg group Fourier transform}\label{sec:Heisenberg_Fourier}

Since $\bar{N}MAN$ is open and dense in $G$, the restriction from $G$ to $\bar{N}$ defines an embedding of $I(\nu)$ into $C^\infty(\bar{N})$, the so-called \emph{non-compact picture of $\pi_\nu$}.

Following \cite{FWZ22} we further use the
notation $V_1=\bar{\frakn}_1$ and note that $V_1$ has a canonical
complex structure given by the complex structure for the Hermitian symmetric
subpair $(\frakm, \frakm\cap \frakk)$ of $(\frakg, \frakk)$. We can identify $\bar{\frakn}=\bar{\frakn}_1\oplus \bar{\frakn}_2=V_1\oplus\RR F$ with $V_1\times\RR$ by
$$ (v,t)\mapsto v+tF \qquad (v\in V_1,t\in\RR). $$

The infinite-dimensional irreducible unitary representations of $\bar{N}$ are parameterized by their central character in $i\bar{\frakn}_2^*$. We identify ${\frakn}_2^*$ with $\RR^\times$ by $\lambda\mapsto\lambda(F)$. Write $\sigma_\lambda$ for the representation with central character $-i\lambda$ which we realize on a Fock space $\calF_\lambda(V_1)$ of (anti-)holomorphic (depending on the sign of $\lambda$) functions on $V_1$ which are square-integrable with respect to a Gaussian measure.

The Heisenberg group Fourier transform is the unitary isomorphism
$$ \calF:L^2(\bar{N}) \to \int_{\RR^\times}\calF_\lambda(V_1)\otimes\calF_\lambda(V_1)^*\,d\lambda $$
given by
$$ \calF u(\lambda) = \sigma_\lambda(u) = \int_{\bar{N}}u(\bar{n})\sigma_\lambda(\bar{n})\,d\bar{n} \qquad (u\in L^1(\bar{N})\cap L^2(\bar{N})), $$
where we identify the Hilbert space tensor product $\calF_\lambda(V_1)\otimes\calF_\lambda(V_1)^*$ with the space of Hilbert--Schmidt operators on $\calF_\lambda(V_1)$. More precisely, using the Hilbert--Schmidt norm $\|T\|_{\operatorname{HS}}=\tr(TT^*)^{\frac{1}{2}}$ we have
\begin{equation}\label{eq:PlancherelHeisenberg}
	\|u\|_{L^2(\bar{N})}^2 = \const\times\int_{\RR^\times}\|\sigma_\lambda(u)\|_{\operatorname{HS}}^2|\lambda|^{d_1}\,d\lambda \qquad (u\in L^2(\bar{N})),
\end{equation}
the constant only depending on the normalization of measures and the dimension $d_1=\dim_\CC V_1$. By \cite[Corollary 3.5.3]{Fra22} the Fourier transform can be extended to a continuous linear map
$$ \calF:\calS'(\bar{N}) \to \calD'(\RR^\times)\otimeshat\Hom(\calF_\lambda(V_1)^\infty,\calF_\lambda(V_1)^{-\infty}) $$
which is injective on $I(\nu)$ for $\Re\nu>-\rho$.

\subsection{The metaplectic representation}

The real space $V_1$ is naturally a symplectic vector space with symplectic form $\omega$ given by
$$ [v,w]=\omega(v,w)F \qquad (v,w\in V_1=\bar{\frakn}_1). $$
Let $\Sp(V_1,\omega)$ denote its symplectic group with Lie algebra $\sp(V_1,\omega).$ For $\lambda \in \RR^\times$ let $\omega_{\met,\lambda}$ be the unique projective representation of $\Sp(V_1,\omega)$ on $\calF_{\lambda}(V_1)$, such that $$\sigma_\lambda(gv,t)=\omega_{\met,\lambda}(g)\circ \sigma_\lambda(v,t)\circ \omega_{\met,\lambda}(g)^{-1}.$$
This representation is called \emph{metaplectic representation} and its equivalence class only depends on the sign of $\lambda$. Let $d\omega_{\met,\lambda}$ be the derived $\sp(V_1,\omega)$-representation.
In particular, $d\omega_{\met,\lambda}$ is given by skew symmetric holomorphic resp. anti-holomorphic differential operators on $V_1$ for $\lambda>0$ resp. $\lambda <0$. The underlying Harish-Chandra module is the space $\calP(V_1)$ of holomorphic resp. antiholomorphic polynomials on $V_1$.
We recall the main result of \cite{FWZ22}, which gives the explicit decomposition of $d\omega_{\met,\lambda}$ restricted to $\frakm$, the Lie algebra of $M$. Let therefore $L:=K\cap M$ with Lie algebra $\frakl$ and let $\frakt_\frakl$ be a Cartan subalgebra of $\frakl^\CC$.
\begin{theo+}[\cite{FWZ22} Theorem~2.2]\label{thm:M_decomposition}
	Let $\lambda>0$ and assume $\frakg\not\simeq\su(p,q),\sp(n,\RR)$. Then the restriction $d\omega_{\met,\lambda}|_\frakm$ of the metaplectic representation of $\sp(V_1,\omega)$ to $\frakm$ decomposes as
	\begin{equation*}
		\label{m-deco}
		d\omega_{\met,\lambda}|_\frakm= \bigoplus_{k=0}^\infty \tau_{-k\delta_0 -\frac 12 \zeta_0},
	\end{equation*}
	where $\tau_\mu$ denotes the unitary highest weight representation of $\frakm$ with highest weight $\mu\in\frakt_\frakl^*$, $\delta_0$ is the lowest root of $V_1$ and $\zeta_0$ is the central character of $\fl$ obtained by restriction of the trace of the defining complex linear action of $\fraku(V_1)\subseteq \fsp(V_1,\omega)$ on $V_1$ to $\ft_{\fl}$.
\end{theo+}
To ease notation, we denote the $k$-th part in the decomposition of $d\omega_{\met,\lambda}|_\fm$ resp. $\omega_{\met,\lambda}|_M$, in the sense of the theorem above, by $d\omega_{\met,\lambda,k}$ resp. $\omega_{\met,\lambda,k}$.

\subsection{The Fourier transform of intertwining operators}

From Theorem~\ref{thm:M_decomposition} it follows that the Fock space $\calF_{\lambda}(V_1)$ decomposes into the direct sum of representation spaces for the unitary highest weight representations $\tau_{-k\delta_0-\frac{1}{2}\zeta_0}$,
\begin{equation}\label{eq:Fock_decomposition}
	\calF_{\lambda}(V_1)=\bigoplus_{k\geq 0}\calF_{\lambda,k}(V_1).
\end{equation}
We denote by $P_k:\calF_{\lambda}(V_1)\to\calF_{\lambda,k}(V_1)$ the orthogonal projections.
Since the Fourier transform is injective on $I(\nu)$ and $I(-\nu)$ for $\Re(\nu)\in (-\rho,\rho)$, we have that the Fourier transform $\widehat{A}_\nu$ of the intertwining operator $A_\nu$, which is given by $\widehat{A}_\nu \sigma_\lambda(u):=\sigma_\lambda(A_\nu u)$, is a diagonal operator with respect to the decomposition \eqref{eq:Fock_decomposition}.
\begin{theo+}[\cite{FWZ22} Theorem~3.1]\label{thm:eigenvalues}
	For $\frakg \not\simeq \su(p,q),\sp(n,\RR)$ the operator $\widehat{A}_\nu$ is given by
	$$\widehat{A}_\nu=\const \times\abs{\lambda}^{-\nu} \sum_{k\geq 0}a_k(\nu)\cdot P_k,$$
with positive constant only depending on the structure constants $\rho=d_1+1$ and $a_1$ (see Table~\ref{tab:2}) and
$$a_k(\nu)= \frac{2^\nu\left(\frac{-\nu+a_1+1}{2}\right)_k\Gamma(\frac{\nu-\rho+a_1+2}{2})\Gamma(\frac{\nu}{2})}{\Gamma(\frac{-\nu+\rho}{2})\Gamma(\frac{\nu+a_1+1}{2}+k)}.$$
\end{theo+}

Combined with \eqref{eq:PlancherelHeisenberg}, Theorem~\ref{thm:eigenvalues} shows the following identity for the invariant Hermitian form $\langle\cdot,\cdot\rangle_\nu$ from \eqref{eq:DefInvForm}:
\begin{equation}\label{eq:FourierTransformInvariantForm}
	\langle u,u\rangle_\nu = \const\times\sum_{k\geq0}a_k(\nu)\int_{\RR^\times}\|P_k\circ\sigma_\lambda(u)\|_{\operatorname{HS}}^2|\lambda|^{d_1-\nu}\,d\lambda \qquad (u\in I(\nu)).
\end{equation}
In particular, at the endpoint of the complementary series $\nu=\nu_0=a_1+1$ we have $a_k(\nu_0)=0$ for all $k\neq 0$, so
\begin{equation}\label{eq:FourierTransformInvariantFormNtm}
	\langle u,u\rangle_{\nu_0} = \const\times\,\,a_0(\nu_0)\int_{\RR^\times}\|P_0\circ\sigma_\lambda(u)\|_{\operatorname{HS}}^2|\lambda|^{d_1-\nu_0}\,d\lambda \qquad (u\in I(\nu)).
\end{equation}

\begin{remark}
	For $\frakg=\su(p,q)$ a similar decomposition as in Theorem~\ref{thm:M_decomposition} is obtained in \cite[Theorem~2.3]{FWZ22} and is also given by a multiplicity free direct sum of unitary highest weight representations. The parametrization is slightly different, so we omit the statement. The eigenvalues of the intertwining operator in this case are found in \cite[Theorem~3.1]{FWZ22}.
	
	In the symplectic case $\frakg=\sp(n,\RR)$, we have $\frakm=\sp(V_1,\omega)$ and the Fock space decomposes into two non-equivalent irreducible subrepresentations, the even and the odd part of the metaplectic representations. The eigenvalues for the standard intertwining operator are given in \cite[Theorem~3.8]{FWZ22}.
	
	However, the results for $\su(p,q)$ and $\sp(n,\RR)$ will not be used in the present paper.
\end{remark}

\section{Associated varieties}\label{sec:AssociatedVarieties}

We compute the associated variety of the representation $\pi_\ntm$ that occurs as proper subrepresentation at the endpoint of the complementary series. Recall that the associated variety $\calV(\pi)$ of an irreducible unitary representation $\pi$ of $G$ with annihilator $\operatorname{Ann}(\pi)\subseteq\calU(\frakg^\CC)$ is the subvariety of $(\frakg^\CC)^*$ corresponding to the graded ideal
$$ \operatorname{gr}\operatorname{Ann}(\pi)\subseteq\operatorname{gr}\calU(\frakg^\CC)\simeq S(\frakg^\CC)\simeq\CC[(\frakg^\CC)^*]. $$
By \cite[Corollay 4.7]{Vog91}, $\calV(\pi)$ is the closure of a single nilpotent coadjoint orbit. Further, by \cite[Satz 3.2~(b) and Korollar 5.4]{BK76} its dimension equals the Gelfand--Kirillov dimension of $\calU(\frakg^\CC)/\operatorname{Ann}(\pi)$ which in turn equals twice the Gelfand--Kirillov dimension of $\pi$ by \cite[Corollary 4.7]{Vog78}.

The strategy is to compute the Gelfand--Kirillov dimension of $\pi_\ntm$ by studying the growth of the dimensions of $K$-types, and to compare it with the dimensions of nilpotent coadjoint orbits. It turns out that the associated variety always is the closure of a next-to-minimal nilpotent coadjoint orbit $\calO_\ntm\subseteq(\frakg^\CC)^*$. The computation of the Gelfand--Kirillov dimension of $\pi_\ntm$ is done uniformly for all cases, but for the comparison with coadjoint orbits we carry out a case-by-case analysis.

\subsection{The Gelfand--Kirillov dimension}

Assume that $\frakg\not\simeq\sp(n,\RR),\su(p,q)$, the case $\su(p,q)$ is treated separately in Section~\ref{eq:AssVarSU(p,q)}. By Theorem~\ref{thm:CS_submodule} the $K$-types of $\pi_\ntm$ are $W_{\mu,\mu,\ell}$ with $\mu\geq|\ell|$, $\mu\equiv\ell\mod2$. Note that $\dim W_{\mu,\mu,\ell}$ is independent of $\ell$ since $\ell$ only parameterizes the action of the center of $\frakk$ which is by a scalar. By the Weyl Dimension Formula, the function $\mu\mapsto\dim W_{\mu,\mu,\ell}$ is a polynomial. We first compute its degree.

\begin{lemma}
	The degree of the polynomial $\mu\mapsto\dim W_{\mu,\mu,\ell}$ is $a_1+4b_1+2$.
\end{lemma}

\begin{proof}
	Recall the Weyl Dimension Formula
	$$ \dim W_{\mu,\mu,\ell} = \prod_{\alpha\in\Delta^+(\frakk^\CC,\frakt^\CC)}\frac{\langle\lambda+\rho,\alpha\rangle}{\langle\rho,\alpha\rangle}, $$
	where $\frakt\subseteq\frakk$ is a Cartan subalgebra and $\lambda\in(\frakt^\CC)^*$ denotes the highest weight of $W_{\mu,\mu,\ell}$ with respect a system $\Delta^+(\frakk^\CC,\frakt^\CC)$ of positive roots. It follows that $\dim W_{\mu,\mu,\ell}$ is a polynomial in $\mu$ whose degree is the number of positive roots which are \emph{not} orthogonal to $\lambda$. Passing to the restricted root system
	$$ \{\pm2\alpha_1,\pm2\alpha_2\}\cup\{\pm\alpha_1\pm\alpha_2\}\quad\Big[\cup\{\pm\alpha_1,\pm\alpha_2\}\Big] $$
	with multiplicities $1$, $a_1$ and $2b_1$, we can express $\lambda$ as $\mu(\alpha_1+\alpha_2)+\ell\alpha_0$. The positive restricted roots not orthogonal to $\alpha_1+\alpha_2$ are $2\alpha_1$ and $2\alpha_2$, each with multiplicity $1$, $\alpha_1+\alpha_2$ with multiplicity $a_1$, and possibly $\alpha_1$ and $\alpha_2$, each with multiplicity $2b_1$. Adding up the multiplicities shows the claim.
\end{proof}

\begin{proposition}\label{prop:GKDimGeneral}
	The Gelfand--Kirillov dimension of $\pi_\ntm$ is $a_1+4b_1+4$.
\end{proposition}

\begin{proof}
	Let $V_0=W_{0,0,0}$ and $V_n=\calU_n(\frakg^\CC)V_0$, where $\{\calU_n(\frakg^\CC)\}_{n\geq0}$ is the natural filtration of the universal enveloping algebra $\calU(\frakg^\CC)$ of $\frakg^\CC$. From \cite[Theorem 3.1]{Zha22} it follows that
	$$ V_n = \bigoplus_{\substack{|\ell|\leq\mu\leq n\\\mu\equiv\ell\mod2}}W_{\mu,\mu,\ell}, $$
	and hence
	\begin{equation*}
		\dim V_n = \sum_{\mu=0}^n\sum_{\substack{\ell=-\mu\\\ell\equiv\mu\mod2}}^\mu\dim W_{\mu,\mu,\ell} \sim \sum_{\mu=0}^n\sum_{\substack{\ell=-\mu\\\ell\equiv\mu\mod2}}^\mu\mu^{a_1+4b_1+2} = \sum_{\mu=0}^n(\mu+1)\mu^{a_1+4b_1+2} \sim n^{a_1+4b_1+4}.\qedhere
	\end{equation*}
\end{proof}

Now we know that the associated variety of $\pi_\ntm$ is the closure of a finite union of nilpotent coadjoint orbits of dimension $2\cdot\GKDIM(\pi_\ntm)=2(a_1+4b_1+4)$. In each of the cases, we show that there is a unique nilpotent orbit $\calO_\ntm$ whose dimension is minimal among all next-to-minimal orbits and equal to $2(a_1+4b_1+4)$. It follows that $\calO_\ntm$ is the unique nilpotent coadjoint orbit of dimension $2\cdot\GKDIM(\pi_\ntm)$, so the associated variety of $\pi_\ntm$ equals $\overline{\calO_\ntm}$.

\subsection{The case $\frakg=\so(2,n)$}\label{sec:AssVarSO(2,n)}

We have $a_1=n-4$ and $b_1=0$, so
$$ \GKDIM(\pi_\ntm) = a_1+4b_1+4 = n. $$
By \cite[Theorem 6.2.5]{CMG93}, there are two next-to-minimal nilpotent coadjoint orbits in $\frakg^\CC\simeq\so(n+2,\CC)$, and they are associated with the partitions $(2^4,1^{n-6})$ and $(3^1,1^{n-1})$. Using \cite[Corollary 6.1.4]{CMG93}, we find that the dimension of the orbit associated with $(2^4,1^{n-6})$ is $2(2n-6)$ while for $(3^1,1^{n-1})$ it equals $2n$. It follows that the associated variety of $\pi_\ntm$ is the next-to-minimal nilpotent coadjoint orbit associated with the partition $(3^1,1^{n-1})$.

\begin{remark}
	In \cite{ZH97}, a next-to-minimal representation of $\upO(p,q)$ is constructed under the assumption that $\min(p,q)\geq4$. The associated variety of this representation corresponds to the partition $(2^4,1^{p+q-8})$. The assumption $\min(p,q)\geq4$ is necessary for the existence of such a representation, because it is equivalent to the nilpotent orbit associated with $(2^4,1^{p+q-8})$ having real points. This implies that for $\min(p,q)=2$ there is no irreducible unitary representation of $G=\upO(p,q)$ with associated variety equal to the next-to-minimal orbit associated with the partition $(2^4,1^{n-6})$. In this sense, our representation $\pi_\ntm$ provides a replacement of the next-to-minimal representation of \cite{ZH97} for the case $\min(p,q)=2$.
\end{remark}

\subsection{The case $\frakg=\so^*(2n)$}

Here $a_1=2$ and $b_1=n-4$, so
$$ \GKDIM(\pi_\ntm) = a_1+4b_1+4 = 4n-10. $$
Since $\frakg^\CC=\so(2n,\CC)$, the discussion in Section~\ref{sec:AssVarSO(2,n)} shows that there are precisely two next-to-minimal nilpotent coadjoint orbits, and their dimensions are $2(4n-10)$ and $2(2n-2)$, respectively. The first one belongs to the partition $(2^4,1^{2n-8})$, and its closure is the associated variety of $\pi_\ntm$.

\subsection{The case $\frakg=\frake_{6(-14)}$}

In this case $a_1=4$ and $b_1=2$, so
$$ \GKDIM(\pi_\ntm) = a_1+4b_1+4 = 16. $$
By \cite[table on p.~129]{CMG93}, there is a unique next-to-minimal nilpotent coadjoint orbit in $\frakg^\CC\simeq\frake_6(\CC)$ and it has dimension $32=2\cdot16$ and Bala--Carter label $2A_1$. This shows that the associated variety of $\pi_\ntm$ is equal to the closure of this orbit.

\subsection{The case $\frakg=\frake_{7(-25)}$}

Here $a_1=6$ and $b_1=4$, so
$$ \GKDIM(\pi_\ntm) = a_1+4b_1+4 = 26. $$
By \cite[table on p.~130]{CMG93}, there is a unique next-to-minimal nilpotent coadjoint orbit in $\frakg^\CC\simeq\frake_7(\CC)$ and it has dimension $52=2\cdot26$ and Bala--Carter label $2A_1$. This shows that the associated variety of $\pi_\ntm$ is equal to the closure of this orbit.

\subsection{The case $\frakg=\su(p,q)$}\label{eq:AssVarSU(p,q)}

This case differs slightly from the other cases in the sense that the dependence of $W_{\mu_1,\mu_2,\ell}$ on $\ell$ is not just by the central character. The root system $\Delta(\frakk^\CC,\frakt^\CC)$ is of type $A_{p-1}\times A_{q-1}$ and we write
$$ \Delta(\frakk^\CC,\frakt^\CC) = \{\pm(e_i-e_j):1\leq i<j\leq p\}\cup\{\pm(f_i-f_j):1\leq i<j\leq q\}. $$
With respect to the positive system
$$ \Delta^+(\frakk^\CC,\frakt^\CC) = \{e_i-e_j:1\leq i<j\leq p\}\cup\{f_i-f_j:1\leq i<j\leq q\} $$
the highest weight of $W_{\mu_1,\mu_2,\ell}$ is
$$ \left(\frac{\mu_1+\ell}{2}e_1-\frac{\mu_1-\ell}{2}e_p\right) + \left(\frac{\mu_2+\ell}{2}f_1-\frac{\mu_2-\ell}{2}f_q\right) $$
and its dimension equals
\begin{multline*}
	\frac{(\mu_1+p-1)(\mu_2+q-1)}{(p-1)(q-1)}{\frac{\mu_1+\ell}{2}+p-2\choose p-2}{\frac{\mu_1-\ell}{2}+p-2\choose p-2}{\frac{\mu_2+\ell}{2}+q-2\choose q-2}{\frac{\mu_2-\ell}{2}+q-2\choose q-2}\\
	\sim \mu_1\mu_2(\mu_1^2-\ell^2)^{p-2}(\mu_2^2-\ell^2)^{q-2}.
\end{multline*}
The $K$-types of $\pi_\ntm$ are given by $W_{\mu_1,\mu_2,\ell}$ with $\mu_1-\mu_2=q-p$ and $|\ell|\leq\mu_1,\mu_2$, $\mu_1\equiv\mu_2\equiv\ell\mod2$. Summing the dimension of $W_{\mu,\mu,\ell}$ over $|\ell|\leq\mu_1,\mu_2\leq n$ with $\mu_1-\mu_2=q-p$ in a similar way as in the previous cases shows that
$$ \GKDIM(\pi_\ntm) = 2p+2q-4. $$

By \cite[Theorem 6.2.5]{CMG93}, the unique next-to-minimal nilpotent coadjoint orbit in $\frakg^\CC\simeq\sl(p+q,\CC)$ is associated to the partition $(2^2,1^{p+q-4})$, and by \cite[Corollary 6.1.4]{CMG93} it has dimension $4(p+q)-8=2(2p+2q-4)$. This implies that the associated variety of $\pi_\ntm$ is equal to the closure of this orbit.

\section{Restriction from $G$ to $\SL(2,\RR)\times M$}\label{sec:BranchingGtoMxSL(2,R)}

Recall that the elements $E\in\frakn_2$, $F\in\bar{\frakn}_2$ and $H\in\fraka$ form an $\sl(2)$-triple that commutes with $\frakm$. Writing $\sl(2,\RR)_A=\operatorname{span}\{E,F,H\}$, we obtain a subalgebra $\sl(2,\RR)_A\oplus\frakm$ and a corresponding subgroup $\SL(2,\RR)_A\times M$. The goal of this section is to understand the restriction of $\pi_\ntm$ to $\SL(2,\RR)_A\times M$. We first make some general observations before specializing to the cases $\frakg=\so(2,n)$ and $\frakg=\frake_{6(-14)}$.

\subsection{The Fourier transformed model of the next-to-minimal representation}\label{sec:FourierTransformForBranching}

To study the restriction of $\pi_\ntm$ to $\SL(2,\RR)_A\times M$, we use the Heisenberg Fourier transform. Assume in what follows that $\frakg\not\simeq\sp(n,\RR),\su(p,q)$.

\begin{lemma}\label{lemma:SO(2,n)_fourier_L^2}
	Viewing $\pi_\ntm$ as a quotient of $I(\nu_0)$, the Heisenberg group Fourier transform $\calF$ induces a unitary (up to a scalar) isomorphism
	$$ \calF: \pi_\ntm \to L^2(\RR^\times, \calF_{\lambda,0}(V_1)^*\otimes\calF_{\lambda}(V_1),\abs{\lambda}^{2b_1+1}\,d\lambda), \quad \calF u(\lambda)=\sigma_\lambda(u). $$
\end{lemma}

Here, the space $L^2(\RR^\times, \calF_{\lambda,0}(V_1)^*\otimes\calF_{\lambda}(V_1),\abs{\lambda}^{2b_1+1}\,d\lambda)$ has to be understood as follows. By working for instance with the Schr\"{o}dinger model of the oscillator representation, one can find a Hilbert space $\calH$ and for every $\lambda\in\RR^\times$ a unitary isomorphism $\calF_\lambda(V_1)\simeq\calH$ such that $\calF_{\lambda,0}(V_1)$ is mapped onto a subspace $\calH_0$. Then $L^2(\RR^\times, \calF_{\lambda,0}(V_1)^*\otimes\calF_{\lambda}(V_1),\abs{\lambda}^{2b_1+1}\,d\lambda)$ corresponds to the space $L^2(\RR^\times,\calH_0^*\otimes\calH,|\lambda|^{2b_1+1}\,d\lambda)$ of $L^2$-functions on $\RR^\times$ with respect to the measure $|\lambda|^{2b_1+1}\,d\lambda$ with values in the Hilbert space $\calH_0^*\otimes\calH$. However, for us it is more convenient to work with functions whose value at $\lambda\in\RR^\times$ is contained in $\calF_{\lambda,0}(V_1)^*\otimes\calF_{\lambda}(V_1)$, so we suppress the isomorphism $\calF_\lambda(V_1)\simeq\calH$ in what follows.

\begin{proof}[Proof of Lemma~\ref{lemma:SO(2,n)_fourier_L^2}]
	Following \cite[Section~4]{FWZ22}, the representation $\pi_\ntm$ is a quotient of the principal series $(\pi_{\nu_0},I(\nu_0))$ for $\nu_0=a_1+1$. 
	Explicitly it is given as the quotient by the kernel of the intertwining operator $A_{\nu_0}:I(\nu_0)\to I(-\nu_0)$.
	Following \cite[Section~1.5]{FWZ22}, the Fourier transform
	$$\calF: I(\nu_0)\to \bigsqcup_{\lambda\in \RR^\times}\Hom(\calF_\lambda(V_1),\calF_{\lambda}(V_1))$$ is defined with distributional dependence on the parameter $\lambda$. 
	The inner product on the quotient $I(\nu_0)/\Ker A_{\nu_0}=\pi_\ntm$ is by \eqref{eq:FourierTransformInvariantFormNtm} explicitly given for $f_1,f_2\in \pi_\ntm$ as
	$$\langle f_1, f_2 \rangle_{\nu_0}=\const \times \int_{\RR^\times} \tr\left(\sigma_\lambda(f_1)\circ P_0 \circ \sigma_{\lambda}(f_2)^\ast\right)\abs{\lambda}^{d_1-\nu_0}\, d\lambda,$$
	where $P_0:\calF_{\lambda}(V_1)\to \calF_{\lambda,0}(V_1)$ is
        the orthogonal projection. Since $d_1=a_1+2b_1+2$ and $\nu_0=a_1+1$, the claim follows.
\end{proof}

To employ the Heisenberg group Fourier transform for the decomposition of the restriction of $\pi_\ntm$ to $\SL(2,\RR)_A\times M$, we need to understand how it behaves with respect to the action of $\SL(2,\RR)_A$ and $M$. The action of a general element of $\SL(2,\RR)_A$ turns out to be quite complicated, but the action of the parabolic subgroup $B:=\exp(\RR H)\exp(\RR F)\subseteq \SL(2,\RR)_A$ of $\SL(2,\RR)_A$ is rather simple:

\begin{lemma}[{see \cite[Proposition~3.5.5]{Fra22}}]\label{lemma:AN_action}
Let $u\in I(\nu)$.
\begin{enumerate}
\item\label{lemma:AN_action1} ($B$-action) For $t\in \RR$ we have
	\begin{align*}
		\sigma_\lambda(\pi_\nu(e^{tH})u) &= e^{(\nu-\rho)t}\delta_{e^t}\circ \sigma_{e^{-2t}\lambda}(u)\circ \delta_{e^{-t}},\\
		\sigma_\lambda(\pi(e^{tF})u) &= e^{-i\lambda t}\sigma_\lambda(u),
	\end{align*}
	where $\delta_s \zeta(z)=\zeta(sz)$.
\item\label{lemma:AN_action2} ($M$-action) For $m\in M$ we have
	$$ \sigma_\lambda(\pi_\nu(m)u) = \omega_{\met,\lambda}(m)\circ\sigma_\lambda(u)\circ\omega_{\met,\lambda}(m)^{-1}. $$
\end{enumerate}
\end{lemma}

By the last formula in Lemma~\ref{lemma:AN_action}, the decomposition of $\pi_\ntm|_M$ is related to the decomposition of the tensor product representation $\omega_{\met,\lambda,0}^*\otimes\omega_{\met,\lambda}|_M$. For this, we specialize to $\frakg=\so(2,n)$ or $\frakg=\frake_{6(-14)}$. Note that $\omega_{\met,\lambda,0}^*\simeq\omega_{\met,-\lambda,0}$ since $\omega_{\met,\lambda}$ and $\omega_{\met,-\lambda}$ are contragredient to each other.

\subsection{The case $G=\SO_0(2,n)$}\label{sec:BranchingSO(2,n)toMxSL(2,R)}

Let $G=\SO_0(2,n)$, $n>4$, then $\frakm=\sl(2,\RR)\oplus \so(n-2)$. To distinguish the copy of $\sl(2,\RR)$ in $M$ from $\sl(2,\RR)_A$, we denote it by $\sl(2,\RR)_M$ and similarly on the group level.

As abstract $M$-representation, the metaplectic representation $\omega_{\met,\lambda}$ and its subrepresentations $\omega_{\met,\lambda,k}$ do not depend on the parameter $\lambda$, but only on its sign. Hence we might suppress the parameter in the following whenever convenient and write
$$ \omega_{\met}^{\sgn(\lambda)}\cong \omega_{\met,\lambda} \qquad \mbox{and} \qquad \omega_{\met,k}^{\sgn(\lambda)}\cong \omega_{\met,\lambda,k}. $$

The decomposition of $\omega_{\met}^+|_M$ is made explicit in \cite[Proposition~B.2]{FWZ22}. Denote by $\eta_k^{\SO(m)}$ the irreducible representation of $\SO(m)$ on the space $\calH^k(\RR^m)$ of homogeneous harmonic polynomials on $\RR^m$ of degree $k$. Moreover, let $\tau_{\mu,\varepsilon}^{\SL(2,\RR)}$, $\mu\in i\RR$, $\varepsilon\in\ZZ/2\ZZ$, be the unitary principal series of $\SL(2,\RR)$, spherical for $\varepsilon=0$ and non-spherical for $\varepsilon=1$. We further write $\tau_k$ ($k\in \ZZ\setminus \{0\}$) for the (limit of) discrete series of $\SL(2,\RR)$ of parameter $k$, where the notation is such that $\tau_k$ is a holomorphic resp. antiholomorphic discrete series for $k\leq-2$ resp. $k\geq2$. Then
\begin{equation}\label{eq:restriction_G_met_to_M}
	\omega_{\met}^+|_M = \bigoplus_{k=0}^\infty\,\omega_{\met,\lambda,k}= \bigoplus_{k=0}^\infty\, \tau_{
		- k -(n-2)}^{\SL(2, \mathbb R)} \boxtimes \eta_{k}^{\SO(n-2)}.
\end{equation}
Since $\omega_{\met}^+$ is contragredient to $\omega_{\met}^-$, this also implies the corresponding decomposition of $\omega_{\met}^-$.

Write $\operatorname{par}(k)\in\ZZ/2\ZZ$ for the parity of $k\in\ZZ$.

\begin{lemma}\label{lemma:tensor_product_met}
	The following tensor product decomposition for representations of $M=\SL(2,\RR)_M\times\SO(n-2)$ holds:
	$$  \omega_{\met,0}^\mp\otimes \omega_{\met}^\pm|_M \simeq \bigoplus_{k\geq 0}\int^\oplus_{i\RR_{\geq 0}} \tau_{\mu,\operatorname{par}(k)}^{\SL(2,\RR)}\boxtimes\eta_k^{\SO(n-2)} \,d\mu\oplus \bigoplus_{l\geq 2} \tau_{\mp l}^{\SL(2,\RR)}\boxtimes\left(\bigoplus_{\substack{k\geq l\\l\equiv k \mod 2}}\eta_k^{\SO(n-2)}\right).$$
\end{lemma}

\begin{proof}
This follows from the classical formulas for tensor products of holomorphic and anti-holomorphic discrete series of $\SL(2,\RR)$ (see e.g. \cite{Repka78}) and the decomposition \eqref{eq:restriction_G_met_to_M}.
\end{proof}

To also obtain information about the action of the other factor $\SL(2,\RR)_A$ on the isotypic components, we make the following observation:

\begin{prop+}\label{prop:SL(2,R)_conjugate}
	The two copies $\SL(2,\RR)_A$ and $\SL(2,\RR)_M$ of $\SL(2,\RR)$ in $G$ are conjugate via an element of $\SO(2)\times \mathrm{S}(\mathrm{O}(2)\times\mathrm{O}(n-2))\subseteq\SO(2)\times \SO(n)\subseteq G$.
\end{prop+}

\begin{proof}
	Note that it suffices to show that the two copies of $\so(2)$ in $\sl(2,\RR)$ are conjugate. Choose for example explicitly $H=E_{1,n+1}+E_{2,n+2}+E_{n+1,1}+E_{n+2,2}$. Then $\frakk\cap(\sl(2,\RR)_A\oplus\frakm)= \so(2)\oplus \so(2)\oplus\so(n-2)$ is realized in diagonal blocks in $\frakg$ in the following way. One copy of $\so(2)$ is spanned by $\diag(X,X,0_{n-2})$ and the other one by $\diag(X,-X,0_{n-2})$, where
$$X=\begin{pmatrix}
		0& 1 \\ -1 & 0
\end{pmatrix}.$$
Such two elements are conjugate via the matrix $g=\diag(1_2,y,1_{n-2})\in\upO(2,n)$, where
$$y=\begin{pmatrix}
		0& 1 \\ 1 & 0
\end{pmatrix},$$
and conjugation with $g$ is trivial on $\SO(n-2)$. Since $n>2$, the product of $g$ with some element in $\upO(n-2)$ of determinant $-1$ does the job. (The explicit matrix realizations of all subgroups can also be found in \cite[Appendix~B.2]{Fra22}.)
\end{proof}

Recall the subgroup $B\subseteq \SL(2,\RR)$. By the classification of the irreducible unitary representations of $B$ (see e.g. \cite{GN47}), there exist exactly two non-equivalent infinite dimensional irreducible unitary representations of $B$. They can be realized on $L^2(\RR_\pm,\abs{\lambda}^{\Re\alpha-1}\,d\lambda)$ for any $\alpha\in\CC$ by the action
\begin{align*}
	(e^{tH}.\varphi)(\lambda) &= e^{(\alpha-2)t}\varphi(e^{-2t}\lambda),\\
	(e^{tF}.\varphi)(\lambda) &= e^{ -i \lambda t}\varphi(\lambda),
\end{align*}
where $\varphi\in L^2(\RR_\pm,\abs{\lambda}^{\Re\alpha-1}\,d\lambda)$ and $t\in\RR$. We denote the equivalence classes by $\sigma_B^+$ and $\sigma_B^-$.

\begin{lemma}\label{lemma:SL(2,R)_to_AN_restriction}
	The restriction of a unitary principal series representation $\tau_{\mu,\varepsilon}^{\SL(2,\RR)}$ ($\mu\in i\RR$, $\varepsilon\in\ZZ/2\ZZ$) or a lowest/highest weight representation $\tau_{\pm k}^{\SL(2,\RR)}$ ($k>0$) of $\SL(2,\RR)$ to $B$ is given by
	$$
	\tau_{\mu,\varepsi}^{\SL(2,\RR)}|_{B}\cong \sigma_B^+\oplus \sigma_B^-, \qquad \tau_{\pm k}^{\SL(2,\RR)}|_{B}\cong \sigma_B^\mp.$$
\end{lemma}

\begin{proof}
	The unitary principal series of $\SL(2,\RR)$ can be realized on $L^2(\RR)$, where the action on a function $f\in L^2(\RR)$ is given by (see for example \cite[Chapter~II]{Kna86})
	$$\begin{pmatrix}
		a & 0 \\ c & a^{-1}
	\end{pmatrix}.f(x)=af(a^2x-ac) \qquad (a\in\RR^\times,c\in\RR).$$
	Applying the Euclidean Fourier transform on the real line proves the statement.\\
	Let $\Pi=\{z=x+iy\in\CC:y>0\}\subseteq \CC$ be the upper half plane. 
	The holomorphic discrete series $\tau_{-k}^{\SL(2,\RR)}$, $k\geq 2$ can be realized on the weighted Bergman space $H^2_k(\Pi)$ on $\Pi$, given by
	$$H^2_k(\Pi)=\calO(\Pi)\cap L^2(\Pi, y^{k-2}\,dx\,dy).$$
	Explicitly $\SL(2,\RR)$ acts on on a function $f\in H^2_k(\Pi)$ by
	$$(\pi_{-k}^{\SL(2,\RR)}(g)f)(z)=(-bz+d)^{-k}f\left(\frac{az-c}{-bz+d}\right),$$
	where $$g=\begin{pmatrix}a & b \\ c & d
	\end{pmatrix}\in \SL(2,\RR).$$	
	For any $k>1$ consider the Laplace-transform $\mathcal{L}$ on $L^2(\RR_+,\abs{\lambda}^{k-1}\,d\lambda)$, given by
	$$\mathcal{L}g(z)=\int_0^\infty g(\lambda)e^{i\lambda z} \, d\lambda. $$
 	By \cite[Theorem~XIII.1.1]{FK94} the Laplace-transform gives a surjective isometry onto $H^2_k(\Pi)$ and it is easily checked that it intertwines the action of $B$ given by $\sigma_B^+$.
	For the anti-holomorphic discrete series the statement follows by taking complex conjugates. For the limits of discrete series we have that $\pi_{-1}^{\SL(2,\RR)}\oplus \pi_{1}^{\SL(2,\RR)}$ is a unitary principal series which restricts to $B$ as $\sigma_B^+\oplus \sigma_B^-$ and the statement follows by a small modification of the argument.
\end{proof}

We can finally combine all the gathered information to obtain the full decomposition of $\pi_\ntm$ restricted to $\SL(2,\RR)_A\times M=\SL(2,\RR)_A\times\SL(2,\RR)_M\times\SO(n-2)$:

\begin{theo+}\label{thm:RestrictionOfSO(2,n)toMxSL2}
The restriction of the next-to-minimal representation of $\SO_0(2,n)$ to $\SL(2,\RR)_A\times \SL(2,\RR)_M\times\SO(n-2)$ is given by
	\begin{multline*}
	\pi_\ntm|_{\SL(2,\RR)\times \SL(2,\RR)\times\SO(n-2)} \simeq \bigoplus_{k=0}^\infty\Bigg(\int_{i\RR_{\geq0}}\tau_{\mu,\operatorname{par}(\mu)}^{\SL(2,\RR)}\boxtimes\tau_{\mu,\operatorname{par}(\mu)}^{\SL(2,\RR)}\,d\mu\\
	\oplus\bigoplus_{\substack{2\leq|\ell|\leq k\\\ell\equiv k\mod2}}\tau_\ell^{\SL(2,\RR)}\boxtimes\tau_\ell^{\SL(2,\RR)}\Bigg)\boxtimes\eta_k^{\SO(n-2)}.
\end{multline*}
\end{theo+}

\begin{proof}
By Lemma~\ref{lemma:SO(2,n)_fourier_L^2} we have
$$ \pi_\ntm \simeq L^2(\RR^\times, \calF_{-\lambda,0}(V_1)\otimes \calF_\lambda(V_1) ,\abs{\lambda}\,d\lambda), $$
so in view of Lemma~\ref{lemma:AN_action}~\eqref{lemma:AN_action2}, the restriction of $\pi_\ntm$ to $M$ can be written as
 \begin{equation*}
 	\pi_\ntm|_M\simeq \left(L^2(\RR_+,\abs{\lambda}\,d\lambda)\boxtimes( \omega_{\met,0}^-\otimes \omega_{\met}^+|_M)\right)\oplus \left( L^2(\RR_-,\abs{\lambda}\,d\lambda)\boxtimes(\omega_{\met,0}^+\otimes \omega_{\met}^-|_M)\right),
 \end{equation*}
by dividing $\RR^\times$ into the positive and negative axes. Here $M$ is acting trivially on $L^2(\RR_\pm,\abs{\lambda}\,d\lambda)$. It follows that the action of $B$ preserves this decomposition and acts on each $M$-isotypic component of $\omega_{\met,0}^\mp\otimes\omega_{\met}^\pm|_M$ by unitary automorphisms. The action on $\varphi \in L^2(\RR^\times, \abs{\lambda}\,d\lambda)$ is explicitly given in Lemma~\ref{lemma:AN_action}~\eqref{lemma:AN_action1} for $\nu=\nu_0=n-3=\rho-2$:
 $$ (e^{tH}\cdot\varphi)(\lambda)=e^{-2t}\varphi(e^{-2t}\lambda), \qquad (e^{tF}\cdot\varphi)(\lambda)=e^{-i\lambda t}\varphi(\lambda),$$
so that
 \begin{equation}\label{eq:ntm_M_decomp}
	\pi_\ntm|_{B\times M}\simeq \left(\sigma_B^+\boxtimes(  \omega_{\met,0}^-\otimes \omega_{\met}^+|_M)\right)\oplus \left(\sigma_B^-\boxtimes(\omega_{\met,0}^+\otimes \omega_{\met}^-|_M )\right).
\end{equation}
Since the subgroups $B$ of the two copies of $\SL(2,\RR)$ in $\SL(2,\RR)\times M$ are conjugate by Proposition~\ref{prop:SL(2,R)_conjugate}, we have by Lemma~\ref{lemma:SL(2,R)_to_AN_restriction}
\begin{multline*}
	\pi_\ntm|_{B\times B\times\SO(n-2)}\simeq  \bigoplus_{k\geq 0}\int_{i\RR_{\geq 0}}\left((\sigma_B^+\oplus \sigma_B^-)\boxtimes (\sigma_B^+\oplus \sigma_B^-)\boxtimes\eta_{k}^{\SO(n-2)}\right)\,d\mu\\
	 \oplus \bigoplus_{\abs{k}\geq 2}\sigma_B^{-\sgn(k)} \boxtimes \sigma_B^{-\sgn(k)}\boxtimes\Bigg(\bigoplus_{\substack{l\geq \abs{k}\\l\equiv k \mod 2}}\eta_l^{\SO(n-2)}\Bigg).
\end{multline*}
	Since even the whole two copies of $\SL(2,\RR)$ are conjugate, and since we know the action of the left copy by \eqref{eq:ntm_M_decomp} and Lemma~\ref{lemma:tensor_product_met}, this implies the theorem by Lemma~\ref{lemma:SL(2,R)_to_AN_restriction}.
\end{proof}

\subsection{The case $G=E_{6(-14)}$}\label{sec:BranchingE6toMxSL(2,R)}

We now study the same branching problem for $G=E_{6(-14)}$. For this, we first fix some notation regarding the exceptional Lie algebra $\frake_{6(-14)}$.

\subsubsection{The subalgebra
	$\fsu(5, 1)
	\subseteq \fe_{6(-14)}$}

Recall again our convention from \cite{FWZ22, Zha22}
that the Heisenberg parabolic
subalgebra is constructed
using the lowest Harish-Chandra
root   $\gamma_1$. More precisely -- in the present case --
let  $\gamma_1<\gamma_2$  be
Harish-Chandra strongly orthogonal roots
for the symmetric pair $(\fg, \fk)
=(\frake_{6(-14)}, \mathfrak{so}(2) +\mathfrak{spin}(10))
$
and $e_{\pm1}$ the root vectors
for $\pm \gamma_1$
with $[e_1, e_{-1}]$ being the corresponding
co-root. Our  $H$ in Section 1 is then $H=e_1+e_{-1}$.
The Levi subalgebra is $\fm=\fsu(5,1)$
and its maximal compact subalgebra is $\fl=\fu(5)=\fk\cap \fm$. 
To specify the relevant  roots
we use the Dynkin diagrams of 
of $\fe_6$ and $\fm^{\mathbb C}$.
The diagram is 
$$
\begin{tikzcd}[arrows=-]
	E_6{:} &[-2em]
	\stackrel{\alpha_1}{\bullet} \arrow[r] &
	\stackrel{\alpha_3}{\circ} \arrow[r] &
	\stackrel{\alpha_4}{\circ} \arrow[r]\arrow[d]
	&\stackrel{\alpha_5}{\circ} \arrow[r] &
	\stackrel{\alpha_6}{\circ} \\
	& && \stackrel{\circ}{\alpha_2}  &&
\end{tikzcd}, 
$$
where the circled roots  $\{\alpha_2, \ldots, \alpha_6\}$
are the compact roots and the black root  $\gamma_1=\alpha_1$ is the non-compact 
lowest root.
The compact and non-compact positive roots are (see e.g. \cite{EHW83})
\begin{align*}
	\Delta_c^+ &= \{\e_i\pm \e_j; 5\ge i>j\ge 1\},\\
	\Delta_n^+ &= \Big\{\dfrac 12 (\sum_{i=1}^5 (-1)^{\nu_i} \e_i -\e_6-\e_7 +\e_8); \text{$\sum_{i=1}^5 {\nu_i} $ is even}\Big\},
\end{align*}
with 
$$
\gamma_1=\alpha_1=\dfrac 12 (
\e_1 -\e_2-\e_3-\e_4 -\e_5 -\e_6 -\e_7 
+ \e_8 ), \quad \alpha_2 =\e_1+\e_2, \quad \alpha_j=\e_{j-1}-\e_{j-2},\,  (3\le j\le 6).
$$

The  Harish-Chandra strongly ortogonal roots
are $\gamma_1
<\gamma_2$, where
$$\gamma_2 =
\dfrac 12 (-\e_1+\e_2+\e_3 +\e_4 -\e_5 
-\e_6 -\e_7+ \e_8).$$
We shall need the opposite
Harish-Chandra roots starting from the highest one.
This is the pair
$
\tilde \gamma_2 >\tilde \gamma_1,
$
with
$$\tilde \gamma_2 =
\dfrac 12 (\e_1+\e_2+\e_3 +\e_4 +\e_5 
-\e_6 -\e_7+ \e_8), \qquad
\tilde \gamma_1 
=\frac 12(-\varepsilon_1-\varepsilon_2 
-\varepsilon_3-\varepsilon_4 +\varepsilon_5 
-\varepsilon_6-\varepsilon_7 +\varepsilon_8).
$$
These formulas can  easily be checked 
since there are five non-compact positive 
roots orthogonal to the highest root $\tilde \gamma_2$, 
and $\tilde \gamma_1$ is the highest one among
them with our given ordering. 

The positive roots of $\fm^{\mathbb C}=\fsl(6, \mathbb C)$ are
\begin{align*}
	\Delta_c^+(\fm^{\mathbb C}) &= \{\e_j- \e_i; 5\ge j>i\ge 2\}\cup\{\e_i+ \e_1; 5\ge i\ge 1 \},\\
	\Delta_n^+(\fm^{\mathbb C}) &= \Big\{\dfrac 12 (\sum_{i=1}^5 (-1)^{\nu_i} \e_i -\e_6-\e_7 +\e_8)\in \Delta_n^+; \text{$(-1)^{\nu_1}-\sum_{i=2}^5 (-1)^{\nu_i} +3=0$}\Big\}.
\end{align*}
The roots $\beta_1=\alpha_2,  \beta_2=\alpha_4, \beta_3=\alpha_5,  \beta_4=\alpha_6$
form a system of simple compact roots, 
and together with $\beta_5=\gamma_2$ we get
a system of simple roots for $\fm^{\mathbb C}$, all orthogonal to $\gamma_1$.
The Dynkin diagram for $\fm^{\mathbb C}$ is now
$$
\begin{tikzcd}[arrows=-]
	\fm^{\mathbb C}{:} &[-2em]
	\stackrel{\beta_1}{\circ} \arrow[r] &
	\stackrel{\beta_2}{\circ} \arrow[r] &
	\stackrel{\beta_3}{\circ} \arrow[r]
	&\stackrel{\beta_4}{\circ} \arrow[r] &
	\stackrel{\beta_5}{\bullet}\end{tikzcd}.
$$
Let $\lambda_2$ and $\lambda_5$
be the corresponding fundamental weights
dual to $\beta_2$ and  $\beta_5$ as
representations of
$\fl^{\mathbb C} =
\fgl(5, \mathbb C)$. Then 
$$
\lambda_2 =\frac 13(2\beta_1 + 4\beta_2
+ 3\beta_3+2\beta_4+\beta_5), \qquad
\lambda_5=
\frac 16 (
\beta_1
+2 \beta_2
+3 \beta_3 
+4 \beta_4
+5 \beta_5 ).
$$
Considered as a character
of $\fl$ and $\fl^{\mathbb C}$, $\lambda_5$
is the fundamental central character of
$\fl^{\mathbb C}$, which
we write as $\mathbb C^{\fl}_{\lambda_5}$.

The Harish-Chandra strongly orthogonal root
for
$\fm=\fsu(5,1)$
is then $\beta_3=\alpha_5$,
and the highest non-compact root is 
$\tilde\gamma_2$; 
here $\fm$
is not of tube type so $\beta_5=\gamma_2$
is the lowest non-compact root orthogonal to $\gamma_1$
and there are three non-compact 
roots between $\beta_5=\gamma_2$ and 
$\tilde\gamma_2$, they are 
$$
\beta_5 + \beta_4, \quad
\beta_2 + (\beta_4+\beta_3), \quad
\beta_2 + (\beta_4+\beta_3+\beta_2). 
$$
The Harish-Chandra
decomposition for  
$\fm=\fsu(5,1)$ is
$$
\fm^{\mathbb C}=\fsl(6, \mathbb C)
=\bar{\CC}^5 + \fgl(5, {\mathbb C})+ \CC^5.
$$

The space $V_1$ has lowest weight  
$$
\delta_0  
=\dfrac 12 (
-\e_1 + \e_2-\e_3-\e_4 -\e_5 
-\e_6 -\e_7 
+ \e_8 ),  
$$
and the central character $\tr \ad_{V_1}$
of
$
\fgl(5, \mathbb C) 
$
is easily found to be
$$
\frac 12 \tr \ad_{V_1} = 3 \lambda_5 =
\frac 12( 5\beta_5  
+ 4\beta_4 + 3 \beta_3 +2 \beta_2+
\beta_1) .
$$

\subsubsection{Tensor products of highest and lowest weight representations of $\fsu(5, 1)$}

The branching of the
metaplectic representation $\omega_{\met}^+$
of $\Sp(V_1,\omega)$ restricted to $M=\SU(5, 1)$
is explicitly obtained in \cite{FWZ22}, and it is a sum
of the holomorphic representations $\omega_{\met,k}^+=\pi_{-k\delta_0 -3\lambda_5}$ of $\fsu(5,1)$ with highest weight ${-k\delta_0 -3\lambda_5}$, $k\ge 0$. As in the previous section, we first find the decomposition of the tensor product representation $\omega_{\met,0}^-\otimes\omega_{\met}^+|_M\simeq\pi_{-3\lambda_5}^*\otimes\pi_{-k\delta_0-3\lambda_5}$.

\begin{proposition}\label{prop:TensorProductForE6}
	\begin{enumerate}
		\item
		The highest weight representations
		$(  \pi_{-k\delta_0 - 3\lambda_5}, 
		\fsu(5, 1))$
		do not belong to the discrete series. They are contained in the continuous
		range of the analytic continuation of the discrete series, i.e., they are not reduction points.
		\item
		The tensor product
		$\pi_{-3\lambda_5}^\ast\otimes\pi_{-k\delta_0
			-3\lambda_5}$ is unitarily equivalent
		to the induced representation
		$$L^2(\SU(5, 1)/{\rm U}(5),-k\delta_0)
		:=\Ind_{{\rm U}(5)}^{\SU(5, 1)}
		(-k\delta_0).  $$
	\end{enumerate}	
\end{proposition}

\begin{proof}
	\begin{enumerate}	
		\item
		By computing the inner product with simple roots 
		we find 
		that the highest weight $-\delta_0$
		of the dual representation $V_1'$ is $-\delta_0 = -\lambda_5 +\lambda_2$.
		(Namely it is the representation $\det^{-1}\boxtimes \wedge^2 \mathbb
		C^5$ of $\fu(5) =\fu(1) + \fsu(5)$,  each factor
		acting on the corresponding factor in the tensor product.)
		The discrete series condition can be easily checked; see e.g. \cite{EHW83}.
		
		\item
		We realize $\pi_{-k\delta_0-3\lambda_5}$ on holomorphic sections of the homogeneous vector bundle over $\SU(5,1)/\upU(5)$ induced by the representation $W_{-k\delta_0}$ of $\upU(5)$ of highest weight $-k\delta_0$. Similarly, we realize $\pi_{-3\lambda_5}^*$ on anti-holomorphic functions on $\SU(5,1)/\upU(5)$. Consider the restriction operator
		$$
		R:\pi_{-k\delta_0-3\lambda_5}\otimes 
		\pi_{-3\lambda_5}^\ast\to 
		C^\infty(\SU(5,1)/\upU(5); -k\delta_0), 
		RF(z)=(1-|z|^2)^3 F(z, z),
		$$
		where
		$C^\infty(\SU(5,1)/\upU(5); -k\delta_0) $
		is the space of  smooth sections of the vector bundle 
		induced by the representation $W_{-k\delta_0}$. A direct computation shows that $R$ is $\SU(5,1)$-intertwining.
		Consider then the corresponding  $L^2$-space  $L^2(\SU(5,1)/\upU(5),-k\delta_0)$. 
		Since $-k\delta_0=k(-4\lambda_5+\lambda_2) + 3k\lambda_5$, this is the space 
of $\odot^k(\wedge^3 T_0^{(1, 0)}(B_5)' \otimes \mathbb C_{3k\lambda_5}^{\fl}
		$-valued
		functions on the bounded symmetric domain $B_5=\SU(5,1)/\upU(5)$ with the square norm
		\begin{equation}
			\begin{split}
				\Vert f\Vert^2 
				&=
				\int_{B_5}
				\langle (1-|z|^2)^{-3k}
				\otimes^k (\wedge^3 B(z, z)^t) f(z), f(z)\rangle 
				\,d\iota(z)\\
				&=\int_{B_5}
				\langle  \otimes^k (\wedge^3 (I-\bar z z^t)) f(z), f(z)\rangle 
				\,d\iota(z), 
			\end{split}
		\end{equation}
		where $B(z, z)^t= (1-|z|^2) (I- \bar z z^t)$
		is the metric on $T_z^{(1, 0)}(B_5)'$ dual to the   
		tangent bundle dual to the Bergman metric on the holomorphic tangent space $T_z^{(1, 0)}(B_5)$ with $T_0^{(1, 0)}(B_5)$ being viewed as a representation of $U(5)$, and
		$$d\iota(z)
		=\frac{dm(z)}
		{(1-|z|^2)^{6}
		}
		$$ is the 
		invariant measure on $B_5$
		with $dm(z)$
		the Lebesgue measure; see e.g. \cite{HLZ04}.
		Now the space 
		$  \pi_{
			-k\delta_0 
			-3\lambda_5}$
		contains all 
		$W_{-k\delta_0}$-valued holomorphic 
		polynomials $f_1$, 
		and    $ \pi_{-3\lambda_5}$
		all scalar holomorphic polynomials $f_2$, 
		since $ \pi_{-k\delta_0-3\lambda_5}$ are not reduction
		points in the analytic continuation of the
		holomorphic discrete series. 
		We prove that
		$F=R(f_1\otimes \overline{f_2})$ is in the space
		$  L^2(G/K,-k\delta_0)$.
		Indeed
		we have
		\begin{align*}
			\langle F, 
		F  \rangle
		&=
		\int_{B_5}
		\langle  \otimes^k (\wedge^3 (I-z z^\ast)) (1-|z|^2)^3 
		f_1(z) \overline{f_2}(z), (1-|z|^2)^3 f_1(z)\overline{f_2}(z)\rangle\,d\iota(z)\\
		&=\int_{B_5}
		\langle  \otimes^k (\wedge^3 (I-z z^\ast)) 
		f_1(z) \overline{f_2}(z), f_1(z)\overline{f_2}(z)\rangle 
		\,dm(z).
		\end{align*}
		Now the polynomials $f_1, f_2$ are bounded on $B_5$ as
		well as the matrix $\otimes^k (\wedge^3 (I-z z^\ast)) $,
		and $B_5$ is of finite Lebesgue measure. Thus
		$  \langle F, 
		F  \rangle <\infty    $ and $F \in  L^2(B_5,-k\delta_0)$.
		The rest of the argument is done by abstract argument
		by using the polar decomposition of
		the (unbounded) densely defined closed operator $R$
		with dense image; see e.g. \cite{Zha01} for details about this technique.\qedhere
	\end{enumerate}
\end{proof}

Next, we determine the decomposition of $L^2(\SU(5,1)/\upU(5),-k\delta_0)$ into irreducible representations of $\SU(5,1)$. The representations that occur are unitary principal series of $\SU(5,1)$, so we fix a (minimal) parabolic subgroup $P_M=M_MA_MN_M$ of $M=\SU(5,1)$. Then $M_M$ is a double cover of $\upU(4)$ whose irreducible representations are parameterized by tuples $(\nu_1,\nu_2,\nu_3,\nu_4)$ with $\nu_j\in\frac{1}{2}\ZZ$ and $\nu_i-\nu_j$ for all $1\leq i<j\leq4$. Moreover, the irreducible unitary characters of $A_M$ are parameterized by $\mu\in i\RR$. We denote the corresponding parabolically induced representation of $\SU(5,1)$ by $\tau_{(\nu_1,\nu_2,\nu_3,\nu_4),\mu}^{\SU(5,1)}$.

\begin{proposition}\label{prop:DecompositionL2SU(5,1)/U(5)}
	For every $k\geq0$ the representation $L^2(\SU(5,1)/\upU(5),-k\delta_0)$ of $\SU(5,1)$ and its dual $L^2(\SU(5,1)/\upU(5),-k\delta_0)^*$ decompose into irreducible representations as follows:
	\begin{align*}
		L^2(\SU(5,1)/\upU(5),-k\delta_0) &\simeq \bigoplus_{m=0}^k\int_{i\RR_{\geq0}}\tau_{(\frac{k}{2},\frac{k}{2}-m,-\frac{k}{2},-\frac{k}{2}),\mu}^{\SU(5,1)}\,d\mu,\\
		L^2(\SU(5,1)/\upU(5),-k\delta_0)^* &\simeq \bigoplus_{m=0}^k\int_{i\RR_{\geq0}}\tau_{(\frac{k}{2},\frac{k}{2},\frac{k}{2}-m,-\frac{k}{2}),\mu}^{\SU(5,1)}\,d\mu.
	\end{align*}
\end{proposition}

\begin{proof}
	We first show that the highest weight 
	$-k\delta_0$
	is not the lowest $K$-type
	of any discrete series.
	In the standard notation for 
	$\fsl(6, \mathbb C)$,
	the $\fgl(5,\CC)$-highest weight 
	$-k\delta_0$ is 
	\begin{equation}\label{eq:kdelta0inStandardNotation}
	-k\delta_0 
	=k( -\lambda_5 +\lambda_2) 
	=\frac k2 (1, 1, -1, -1, -1, 1). 
	\end{equation}
	With the existing ordering 
	of the roots of 
	$\fsl(6, \mathbb C)$
	we have that a discrete series of 
	$\SU(5, 1)$ is determined 
	by a non-singular 
	highest weight $
	\mu:= (\mu_1, \mu_2, \mu_3, \mu_4, \mu_5, \mu_6)$
	with 
	$\mu_1>  \cdots > \mu_5 >\mu_6$
	with lowest $K$-type 
	$$
	\mu +\rho_{\fsl(6, \mathbb C)} -2\rho_{\fgl(5,\CC)}
	=\frac 12 (2\mu_1 -3, 2\mu_2-1, 
	2\mu_3+1,   2\mu_4+1, 
	2\mu_5+3,   2\mu_6-5).
	$$
	Here   $\mu_1>  \cdots >\mu_5 >\mu_6$ are all
	half-integers and $\sum_j\mu_j=0$.
	That 
	$L^2(\SU(5,1)/\upU(5),-k\delta_0)$
	has a discrete series  representation 
	is equivalent 
	to that 
	a permutation $s(-k\delta_0)$ of 
	the weight $-k\delta_0$
	by a Weyl group element $s\in S_6$
	produces the above lowest $K$-type:
	$$
	s(-k\lambda_0) 
	=\frac 12 (
	2\mu_1 -3, 2\mu_2-1, 
	2\mu_3+1, 
	2\mu_4+1, 
	2\mu_5+3,   2\mu_6-5 
	). 
	$$
	The only possible choice is then 
	$$
	2\mu_1 -3=k, 2\mu_2-1=k, 
	2\mu_3+1=k, 
	2\mu_4+1=-k, 
	2\mu_5+3=-k,   2\mu_6-5=-k, 
	$$
	i.e., 
	$$
	\mu_1 =\frac 12(k+3), \mu_2=\frac 12(k+1), 
	\mu_3=\frac 12(k-1), 
	\mu_4=\frac12(-k-3), 
	\mu_5=\frac 12(-k-5),  \mu_6=\frac 12(-k+5). 
	$$
	But then $\mu$ is not dominant, thus 
	$L^2(\SU(5,1)/\upU(5),-k\delta_0)$ has no discrete spectrum.
	
	By the Plancherel formula for $\SU(5,1)$, this implies that $L^2(\SU(5,1)/\upU(5),-k\delta_0)$ is equivalent to the multiplicity-free direct integral of unitary principal series induced from those representations of $M_M$ that appear in the restriction of the representation $W_{-k\delta_0}$ of $\upU(5)$. In view of \eqref{eq:kdelta0inStandardNotation} and the standard branching rules for the pair $(\upU(5),\upU(4)$, this implies the claim for $L^2(\SU(5,1)/\upU(5),-k\delta_0)$. Passing to the dual representation and using that $(\tau_{(\nu_1,\nu_2,\nu_3,\nu_4),\mu}^{\SU(5,1)})^*\simeq\tau_{(-\nu_4,-\nu_3,-\nu_2,-\nu_1),-\mu}$ finishes the proof.
\end{proof}

This finally allows us to prove the following result about the restriction of the next-to-minimal representation $\pi_\ntm$ of $G=E_{6(-14)}$ to $\SL(2,\RR)\times\SU(5,1)$:

\begin{theo+}\label{thm:E6_MxSL(2)}
	The restriction of the next-to-minimal representation $\pi_\ntm$ of $E_{6(-14)}$ to $\SL(2, \mathbb{R})\times\SU(5,1)$ contains for every $k>0$ and every $0<m\leq k$ a representation of the form
	$$
	\int_{i\RR_{\geq0}}\tau_{-l}^{\SL(2,\RR)}\boxtimes\tau_{(\frac{k}{2},\frac{k}{2}-m,-\frac{k}{2},-\frac{k}{2}),\mu}\,d\mu \oplus \int_{i\RR_{\geq0}}\tau_{l}^{\SL(2,\RR)}\boxtimes\tau_{(\frac{k}{2},\frac{k}{2},m-\frac{k}{2},-\frac{k}{2}),\mu}\,d\mu
	$$
	for some $l\geq1$ as a direct summand.
\end{theo+}

\begin{proof}
	Recall the subgroup $B\subseteq \SL(2,\RR)_A$ and its equivalence classes of infinite dimensional unitary irreducible representations $\sigma_B^\pm$.
	Following the same line of argumentation as in Section~\ref{sec:BranchingSO(2,n)toMxSL(2,R)}, we have that 
	$$
	\pi_\ntm|_{B\times M} \simeq \left(\sigma_B^+\boxtimes(\omega_{\met,0}^-\otimes \omega_{\met}^+|_M)\right)\oplus\left(\sigma_B^-\boxtimes(\omega_{\met,0}^+\otimes \omega_{\met}^-|_M)\right)
	$$
	which is isomorphic to 
	$$ \bigoplus_{k\geq 0}\left(\sigma_B^+\boxtimes L^2(\SU(5,1)/{\rm U}(5),-k\delta_0) \right) \oplus \left(\sigma_B^-\boxtimes L^2(\SU(5,1)/{\rm U}(5),-k\delta_0)^* \right) $$
	by Proposition~\ref{prop:TensorProductForE6}. Decomposing $L^2(\SU(5,1)/\upU(5),-k\delta_0)$ and its dual using Proposition~\ref{prop:DecompositionL2SU(5,1)/U(5)} shows that this is in turn isomorphic to
	\begin{equation*}
		\bigoplus_{k\geq 0}\bigoplus_{m=0}^k\left(\int_{i\RR_{\geq0}}\sigma_B^+\boxtimes \tau_{(\frac{k}{2},\frac{k}{2}-m,-\frac{k}{2},-\frac{k}{2}),\mu}^{\SU(5,1)}\,d\mu\oplus\int_{i\RR_{\geq0}}\sigma_B^-\boxtimes \tau_{(\frac{k}{2},\frac{k}{2},m-\frac{k}{2},-\frac{k}{2}),\mu}^{\SU(5,1)}\,d\mu\right).
	\end{equation*}
	Now, each representation $\tau_{(\frac{k}{2},\frac{k}{2}-m,-\frac{k}{2},-\frac{k}{2}),\mu}$ resp. $\tau_{(\frac{k}{2},\frac{k}{2},m-\frac{k}{2},-\frac{k}{2}),\mu}$ of $\SU(5,1)$ with $k>0$ and $0\leq m \leq k$ occurs precisely once in this decomposition. Since $\SL(2,\RR)_A$ commutes with $\SU(5,1)$, it acts on the corresponding isotypic component $\sigma_B^+\boxtimes\tau_{(\frac{k}{2},\frac{k}{2}-m,-\frac{k}{2},-\frac{k}{2}),\mu}$ resp. $\sigma_B^-\boxtimes\tau_{(\frac{k}{2},\frac{k}{2},m-\frac{k}{2},-\frac{k}{2}),\mu}$. The only unitary representations of $\SL(2,\RR)$ restricting to $\sigma_B^\pm$ are $\tau_{\mp l}$ ($l>0$) by Lemma~\ref{lemma:SL(2,R)_to_AN_restriction}. In the unitary dual of $\SL(2,\RR)$, these representations are separated, so $l$ has to be constant in every single direct integral. This shows the claim.
\end{proof}

\begin{rema+}
The previous statement intentionally excludes the case $m=0$, because in this case the representations $\tau_{(\frac{k}{2},\frac{k}{2}-m,-\frac{k}{2},-\frac{k}{2}),\mu}$ and $\tau_{(\frac{k}{2},\frac{k}{2},m-\frac{k}{2},-\frac{k}{2}),\mu}$ are equal, so
$$ \left(\sigma_B^+\boxtimes\tau_{(\frac{k}{2},\frac{k}{2},-\frac{k}{2},-\frac{k}{2}),\mu}\right)\oplus\left(\sigma_B^-\boxtimes\tau_{(\frac{k}{2},\frac{k}{2},-\frac{k}{2},-\frac{k}{2}),\mu}\right) \simeq \left(\sigma_B^+\oplus\sigma_B^-\right)\boxtimes\tau_{(\frac{k}{2},\frac{k}{2},-\frac{k}{2},-\frac{k}{2}),\mu} $$
and it is not clear which representation of $\SL(2,\RR)_A$ that restricts to $\sigma_B^+\oplus\sigma_B^-$ acts on the first factor.
\end{rema+}

\begin{rema+}
At this stage, it is not clear how to determine the parameter $l$ for every $k>0$ and $0<m \leq k$. One idea would be to use a subgroup $\SL(2,\RR)_M$ of $M=\SU(5,1)$ isomorphic to $\SU(1,1)\simeq\SL(2,\RR)$ and conjugate to $\SL(2,\RR)_A$ inside $G$. The restriction of $\tau_{(\frac{k}{2},\frac{k}{2}-m,-\frac{k}{2},-\frac{k}{2}),\mu}$ to this subgroup $\SL(2,\RR)_M$ might contain some discrete series representations of $\SL(2,\RR)$ which might be linked to the parameter $l$ by conjugating the subgroup to $\SL(2,\RR)_A$. However, since the restriction of $\tau_{(\frac{k}{2},\frac{k}{2}-m,-\frac{k}{2},-\frac{k}{2}),\mu}$ to $\SL(2,\RR)_M$ will contain several different discrete series representations, it is not clear to us how to use this information.
\end{rema+}


\providecommand{\bysame}{\leavevmode\hbox to3em{\hrulefill}\thinspace}
\providecommand{\MR}{\relax\ifhmode\unskip\space\fi MR }
\providecommand{\MRhref}[2]{\href{http://www.ams.org/mathscinet-getitem?mr=#1}{#2}}
\providecommand{\href}[2]{#2}

\end{document}